\documentclass[12pt]{amsart}
\usepackage{amsfonts,amssymb,amsmath,amscd,amsthm, amstext,bm,color}
\usepackage[colorlinks, citecolor=blue,pagebackref,hypertexnames=false]{hyperref}
\usepackage{enumerate}
\usepackage{epstopdf}
\usepackage{fancyhdr}
\usepackage{geometry}
\usepackage{graphicx}
\usepackage{mathrsfs}
\usepackage{txfonts}
\usepackage{tikz}
\usepackage{url}

\usepackage{comment}
\allowdisplaybreaks

\newtheorem{theorem}{Theorem}[section]

\newtheorem{proposition}[theorem]{Proposition}
\newtheorem{lemma}[theorem]{Lemma}

\newtheorem{remark}[theorem]{Remark}

\newtheorem{definition}[theorem]{Definition}

\numberwithin{equation}{section}

\begin{document}


\title[Sparse domination for singular integral operators in Dunkl setting]{Sparse domination for singular integral operators and their commutators in Dunkl setting with applications}

\author[Y.P. Chen, X.T. Han]{Yanping Chen,
 Xueting Han$^{*}$
}

\thanks{$^*$ Corresponding author}

\address{Yanping Chen, Department of Mathematics, Northeastern University, Shenyang, 110004, Liaoning, China}
\email{yanpingch@126.com}

\address{School of Mathematics and Physics, University of Science and Technology Beijing,
Beijing 100083, China and School of Mathematical and Physical Sciences, Macquarie University, NSW 2109, Australia}
\email{hanxueting12@163.com}




\thanks{
}


\begin{abstract}
In this paper, we establish sparse dominations for the Dunkl-Calder\'on-Zygmund operators and their commutators in the Dunkl setting. As applications, we first define the Dunkl-Muckenhoupt $A_p$ weight and obtain the weighted bounds for the Dunkl-Calder\'on-Zygmund operators, as well as the two-weight bounds for their commutators. Moreover, we also obtain the boundedness of the Dunkl-Calder\'on-Zygmund operators on the extrapolation space of a family of Banach function spaces.
\end{abstract}

\maketitle

\section{Introduction}

Fourier transform plays a key role in classic analysis (see \cite{St2}). Over the last several decades, Dunkl transform has been introduced as a parallel theory to Fourier transform, serving as another important tool in the analysis on Euclidean spaces (see for example \cite{ABFR, AH, ADH, AT, BCG, BCV, BR, Dzi2016, DH2022, HLLW2023, HLLW}).

 Dunkl transform was introduced by Dunkl \cite{Dunkl} under the action of a reflection group. Specifically, given a nonzero vector $\upsilon$, the reflection $\sigma_\upsilon$ with respect to the hyperplane
orthogonal to $\upsilon$ is given by
$$
    \sigma_\upsilon (x)=x-2\frac{\langle x,\upsilon\rangle}{\|\upsilon\|^2}\upsilon,
$$
where $\|\cdot\|$ denotes the Euclidean norm in $\mathbb{R}^N$.
A finite set $R\subseteq\mathbb{R}^N\setminus\{0\}$ is called a root system if $\sigma_{\upsilon} (R)=R$ for every $\upsilon\in R$. In the following, we consider normalized root systems satisfying
$\langle \upsilon,\upsilon \rangle=2$ for all $\upsilon\in R$. Therefore, the reflection $\sigma_{\upsilon}(x)=x-\langle \upsilon, x \rangle \upsilon$ for any $\upsilon\in R$ and $ x\in \mathbb{R}^N$.
A finite group $G$ generated by the set of reflections $\{\sigma_{\upsilon}: \upsilon \in R\}$
 is called the reflection group of the root system $R$. Denote by
$$
    \mathcal{O}(x):=\big\{\sigma (x): \sigma\in G\big\}
$$
 the $G$-orbit of the point $x\in\mathbb{R}^N$.
 The distance between two $G$-orbits $\mathcal O(x)$ and $\mathcal O(y)$ is said to be the Dunkl metric $d$, defined by
$$
    d(x,y):=\min_{\sigma \in G} \|x-\sigma(y)\|.
$$
 It is straightforward to see $d(x,y)\leq\|x-y\|$. The Dunkl ball $\mathcal{O}(B)$ is defined as
$$
    \mathcal{O}(B(x,r)):=\big\{y\in \mathbb{R}^N:d(y,x)<r\big\}.
$$

 Given a root system $R$ and a fixed multiplicity function $\kappa$ which is a nonnegative $G$-invariant function defined on $R$, the $G$-invariant homogeneous weight function $h_{\kappa}$ is defined as
$$
h_{\kappa}(x):=\prod_{\upsilon \in R}|\langle \upsilon, x\rangle|^{\kappa(\upsilon)}.
$$
The associated Dunkl measure is then given by
$$
    d\omega(x):=h_{\kappa}(x)dx=\prod_{\upsilon\in R}|\langle \upsilon, x\rangle|^{\kappa(\upsilon)}
dx.
$$
Denote by $\omega(\mathcal{O}(B))=\int_{\mathcal{O}(B)}d\omega (x)$.
Let $\gamma_{\kappa}=\sum_{\upsilon\in R}\kappa(\upsilon)$, $\mathbf{N}=N+\gamma_{\kappa}$ is the homogeneous dimension associated with the Dunkl setting. Note that $d\omega$ is a doubling measure.

On the space of homogeneous type in the sense of
Coifman and Wiess $(\Bbb R^N, \|\cdot\|, d\omega)$, we recall the space $C^\eta_0(\Bbb
R^N), \eta>0$ which consists the continuous and compactly supported functions $f$ satisfying
$$\|f\|_{\boldsymbol{\eta}}:=\sup\limits_{x\ne y} \frac{|f(x)-f(y)|}{\|x-y\|^{\boldsymbol{\eta}}}<\infty. $$

  In order to study the Littlewood-Paley square functions in the rational Dunkl setting, Dziuba\'{n}ski and Hejna \cite{DH2022CZ} introduced the Dunkl-Calder\'on-Zygmund operator which is
  defined as follows:
    \begin{definition}\label{defT}
    An operator $T: C_0^\eta(\mathbb{R}^N)\rightarrow(C_0^\eta(\mathbb{R}^N))'$ with $\eta>0,$ is said to be a Dunkl-Calder\'on-Zygmund singular integral operator if $K(x,y),$ the kernel of $T$,
    satisfies the following estimates: for some $0<\varepsilon\leqslant 1,$
    \begin{equation}\label{siCZ}
    |K(x,y)|\lesssim \frac1{\omega(B(x,d(x,y)))}
    \end{equation}
    for all $x\not= y;$
    \begin{equation}\label{smooth y3CZ}
    |K(x,y)-K(x,y')|\lesssim \boldsymbol{\Big(\frac{\|y-y'\|}{d(x,y)}\Big)^{\varepsilon}}\frac{1}{\omega(B(x,d(x,y)))}
    \end{equation}
    for $\|y-y'\|\leqslant  d(x,y)/2;$
    \begin{equation}\label{smooth x3CZ}
    |K(x',y)-K(x,y)|\lesssim \boldsymbol{\Big(\frac{\|x-x'\|}{d(x,y)}\Big)^{\varepsilon}}
    \frac1{\omega(B(x,d(x,y)))}
    \end{equation}
    for $\|x-x'\|\leqslant  d(x,y)/2.$

 Moreover,
    \begin{equation*}
    \langle T(f),g\rangle=\int_{\mathbb{R}^N}\int_{\mathbb{R}^N} K(x,y)f(x)g(y)d\omega(x)d\omega(y)
    \end{equation*}
   for $ \operatorname{supp} f \cap \operatorname{supp}g=\emptyset.$

    A Dunkl-Calder\'on-Zygmund singular integral operator is said to be the Dunkl-Calder\'on-Zygmund operator if it extends a bounded operator on $L^2(\mathbb{R}^N,d\omega).$
\end{definition}
A classic example of the Dunkl-Calder\'on-Zygmund operator is Riesz type Dunkl-Calder\'on-Zygmund operators, denoted by $T_R$, first introduced in \cite{THHLL2022} (see also \cite{TanHanLi2023}), whose kernel $K_R$ satisfies the stronger size and smoothness conditions, for some $0<\varepsilon'\leq 1$:
   \begin{equation}\label{si}
    |K_R(x,y)|\lesssim \frac1{\omega(B(x,d(x,y)))}
    \boldsymbol{\Big(\frac{d(x,y)}{\|x-y\|}\Big)^{\varepsilon'}}
    \end{equation}
    for all $x\not= y;$
  \begin{equation}\label{smooth y3}
    |K_R(x,y)-K_R(x,y')|\lesssim \boldsymbol{\Big(\frac{\|y-y'\|}{\|x-y\|}\Big)^{\varepsilon'}}\frac{1}{\omega(B(x,d(x,y)))}
   \end{equation}
    for $\|y-y'\|\leqslant  d(x,y)/2;$
   \begin{equation}\label{smooth x3}
    |K_R(x',y)-K_R(x,y)|\lesssim \boldsymbol{\Big(\frac{\|x-x'\|}{\|x-y\|}\Big)^{\varepsilon'}}
    \frac1{\omega(B(x,d(x,y)))}
    \end{equation}
    for $\|x-x'\|\leqslant  d(x,y)/2.$

In the particular case $\varepsilon' =1$, \eqref{si}-\eqref{smooth x3} are exactly the conditions for the kernel of the Dunkl-Riesz transform $R_j$, $j=1,\ldots,N$ (see \cite[Theorem 1.1]{HLLW}).

 In modern harmonic analysis, sparse domination is a powerful tool and has attracted lots of attention. This idea was first established by Lerner in \cite{Lerner2013} to provide a new proof for the $A_2$ theorem obtained by Hyt\"{o}nen in \cite{Hytonen2012}. Since then, sparse domination has been extensively studied for Calder\'{o}n-Zygmund operator by Lerner and Lacey in \cite{ Lacey2017, Lerner2013-2, Lerner2016}, as well as for many other operators (see for example \cite{CCDO2017, HLSW2021, LaceyMenaRe2017, LaceySpencer2017, SVW2022}).

In this paper, we will explore the sparse domination for the Dunkl-Calder\'on-Zygmund operator, a pointwise domination by localized and averaging dyadic operators:
$$
\sum_{Q \in \mathscr{S}} |f|_{Q} \mathbf{1}_{Q}(x)
$$
Here, $\theta$-sparse family $\mathscr{S}$ $(0<\theta<1)$ is a collection of some dyadic cubes $Q$ (see \cite{HK2012}) in $(\mathbb{R}^N,d\omega)$ such that for every $Q \in \mathscr{S}$, there is a measurable subset $E(Q) \subset Q$ with $\omega(E(Q)) \geq \theta \omega(Q)$ and the sets $\{E(Q)\}_{Q \in \mathscr{S}}$ have only finite overlap. $|f|_{Q}$ is the average of $f$ over the cube $Q$.
Throughout this paper, the constant $\theta$ plays a minor role. Therefore, we will not focus on its precise calculation.

We will also study the sparse domination of the commutator which is defined by
$$
    [b, T](f)(x)=b(x) T(f)(x)-T(b f)(x)
$$
for functions $b\in L_{l o c}^1(\mathbb{R}^N, d \omega)$. The related results are stated as follows.

\begin{theorem}\label{sparseforT}
    Suppose that $T$ is a Dunkl-Calder\'on-Zygmund operator and $\mathscr{D}^d$ is a dyadic system associated to Dunkl metric in $\mathbb{R}^N$.
    \begin{enumerate}
      \item \label{sparseT}
      For any function $f$ with compact support. Then there exists a sparse family $\mathscr{S}^d \subset \mathscr{D}^d$ such that for almost every $x \in \mathbb{R}^N$, we have
    $$
   |Tf(x)|\lesssim \sum_{Q^d \in \mathscr{S}^d} |f|_{Q^d} \mathbf{1}_{Q^d}(x).
    $$
    \item \label{sparsebT}
    Let $b \in L_{\text {loc }}^1(\mathbb{R}^N,d\omega
)$. For every function $f$ with compact support, there exists a sparse family $\mathscr{S}^d \subset \mathscr{D}^d$ such that for almost every $x \in \mathbb{R}^N$, we have
\begin{align*}
 |[b,T]f (x)|
\lesssim \sum_{Q^d \in \mathscr{S}^d}
\left\{
|f|_{Q^{d}}|b(x)-b_{Q^{d}}| +
|(b-b_{Q^{d}})f|_{Q^{d}} \right\}\mathbf{1}_{Q^{d}}(x).
\end{align*}
    \end{enumerate}
    \end{theorem}

\begin{remark}
 Sparse dominations in \eqref{sparseT}-\eqref{sparsebT} also hold for the Riesz type Dunkl-Calder\'on-Zygmund operator and its commutator.
\end{remark}

Indeed, the method of the sparse domination was originally motivated by the study of the quantitative weighted estimates. Later, it is used to obtain two-weight norm inequalities in a number of works (see, for example \cite{AMR, DGKLWY}). As an application of Theorem \ref{sparseforT}, we study the weighted boundedness for $T$ and the two-weight boundedness for the commutator $[b,T]$.

To state our result, let us give the definition of the Dunkl-Muckenhoupt weight function on $(\mathbb{R}^N,d,d\omega)$.
  A nonnegative locally integrable function $u$ is said to belong to $ A^d_p(\mathbb{R}^N,d,d\omega), 1<p<\infty$ if
$$
[u]_{A^d_p}:=\sup _{B \subset \mathbb{R}^N}\Big(\frac{1}{\omega(\mathcal{O}(B))} \int_{\mathcal{O}(B)} u(x) d\omega(x)\Big)\Big(\frac{1}{\omega(\mathcal{O}(B))} \int_{\mathcal{O}(B)} u(x)^{-1/(p-1)} d \omega(x)\Big)^{p-1}<\infty.
$$

It should be pointed out that there are many examples for such weights. For instance, $u(x)=d(x,0)^{\gamma}$ with some $\gamma \in \mathbb{R}$ or the classic Muckenhoupt weight function on $(\mathbb{R}^N, \|\cdot\|,d\omega)$ (see \cite{BS1994} and \cite{DGKLWY}).
We also write
$u(\mathcal{O}(B))=\int_{\mathcal{O}(B)} u (x)d\omega (x)$.

Now, we introduce the spaces of weighted bounded mean oscillation in the Dunkl setting associated to Dunkl metric.
For a weight $u \in A_p^d$, we define the weighted $\mathrm{BMO}^d_u$ space by the set of the functions $b\in L_{loc}^1(\mathbb{R}^N, d \omega)$ satisfying
$$
    \|b\|_{\mathrm{BMO}^d_u}:=\sup _{B \subset \mathbb{R}^N}
    \frac{1}{u(\mathcal{O}(B))} \int_{\mathcal{O}(B)}|b(x)-b_{\mathcal{O}(B)}| d \omega(x)<\infty.
$$
If we replace the Dunkl ball $\mathcal{O}(B)$ by the Euclidean ball $B$ in the definition above, we obtain the norm $\|\cdot\|_{\mathrm{BMO}^{\text{Dunkl}}_u}$, which defines the weighted $\mathrm{BMO}_u^{\text{Dunkl}}$ space with Euclidean metric. When $u=1$, these two spaces reduce to $\mathrm{BMO}^{d}$ and $\mathrm{BMO}^{\text{Dunkl}}$ with the strict inclusion $\mathrm{BMO}^d\subsetneqq \mathrm{BMO}^{\text{Dunkl}}$ (see \cite{JL2023}).

We now present our results as follows.

\begin{theorem}\label{twoweightbound}
Suppose that $T$ is a Dunkl-Calder\'on-Zygmund operator and $u$, $v \in A^d_{p}$, $\vartheta =(\frac{u}{v})^{1/p}$ for $1<p<\infty$. Then,
\begin{enumerate}
  \item \label{weightT}
  $T$ is bounded on $L^{p}(\mathbb{R}^N,d,u d\omega)$;
  \item \label{weightbT}
  If $b \in \mathrm{BMO}^d_{\vartheta}$, then $[b,T]$ is bounded from $L^{p}(\mathbb{R}^N,d,u d\omega)$ to $L^{p}(\mathbb{R}^N,d,v d\omega)$.
\end{enumerate}
\end{theorem}

\begin{remark}\label{BMOdunklremark}
 Assume that $T$ is a Dunkl-Riesz transform $R_j$ with $j=1,\ldots,N$. Han et al. \cite{HLLW} studied the lower and upper bounds of the commutator $[b,R_j]$ with respect to the $\mathrm{BMO}^{\text{Dunkl}}$ and $\mathrm{BMO}^d$ spaces, respectively. Motivated by this, we also have the reverse version of \eqref{weightbT} in Theorem \ref{twoweightbound} for $[b,R_j]$. Let $u$ be a radial $A^d_{p}$ weight and $v \in A^d_{p}$, $\vartheta =(u/v)^{1/p}$. We have the following result:
 $$
 \textit{If $[b,R_j]$ is bounded from $L^{p}(\mathbb{R}^N,d,u d\omega)$ to $L^{p}(\mathbb{R}^N,d,v d\omega)$, then we have $b \in \mathrm{BMO}^{\mathrm{Dunkl}}_{\vartheta}$.}
 $$
 In the case of $u=v=1$, then $\vartheta=1$. The authors in \cite{DH2024} extended the condition of the upper bounds for $[b,R_j]$ to $b\in\mathrm{BMO}^{\mathrm{Dunkl}}$, achieving the characterization of its boundedness via the space $\mathrm{BMO}^{\mathrm{Dunkl}}$ and generalizing the results in \cite{HLLW}.
However, we can not extend the two-weight boundedness condition in \eqref{weightbT} in Theorem \ref{twoweightbound} to $\mathrm{BMO}_{\vartheta}^{\mathrm{Dunkl}}$, because we obtain this via the sparse domination which is formulated in terms of Dunkl cubes.
\end{remark}

Now, we give another application of sparse dominations. Given a family of Banach function spaces $\{X_{\alpha}\}_{\alpha \in \mathcal{I}}$. There are two fundamental extrapolation methods, the intersection and the sum, which can be denoted as
$$
\Delta(\{X_{\alpha}\}_{\alpha \in \mathcal{I}})\quad \text{  and  } \quad
\Sigma(\{X_{\alpha}\}_{\alpha \in \mathcal{I}})
$$
respectively. We define a Hardy-Littlewood maximal operator $M_d$ associated to Dunkl metric on $(\mathbb{R}^N,d,d\omega)$ by
$$
M_d f(x):=\sup _{ B\ni x, B\subset \mathbb{R}^N} \frac{1}{\omega(\mathcal{O}(B))} \int_{\mathcal{O}(B)}|f(y)| \,   d \omega(y).
$$
Denote by $X'$ the dual space of $X$. Then we have the following result.
\begin{theorem}\label{Extrapolation}
Let $\{X_{\alpha}\}_{\alpha \in \mathcal{I}}$ be a family of Banach function spaces, such that $ X_{\alpha} \subset L^{1}_{loc}(\mathbb{R}^N, d, d\omega)$ for every $\alpha \in \mathcal{I}$. Let $T$ be a Dunkl-Calder\'on-Zygmund operator. If $M_d : X_{\alpha} \rightarrow X_{\alpha}$ and
$M_d : X'_{\alpha}\rightarrow X'_{\alpha}$ for every $\alpha \in \mathcal{I}$. Then, $T$ is bounded in the extrapolation space $\Delta(\{X_{\alpha}\}_{\alpha \in \mathcal{I}})$ and $
\Sigma(\{X_{\alpha}\}_{\alpha \in \mathcal{I}})$.
\end{theorem}

Since the size and the smoothness estimates for the Dunkl-Calder\'on-Zygmund operator are involving two metrics, Dunkl metric and Euclidean metric. We can not use the known sparse results directly. We instead make a slight of modification to the key lemma to get the sparse result in the Dunkl setting. To provide the weighted bounds, we first define a Dunkl-Muckenhoupt weight which makes sense in the Dunkl setting and has some basic properties the same as the classic Muckenhoupt $A_p$ weight, especially, reverse H\"{o}lder inequality.

 This following content is organized as follows. We recall some basic definitions and present some useful lemmas in Section \ref{preliminaries}. Then, we give the proof of Theorem \ref{sparseforT} and Theorem \ref{twoweightbound} in Section \ref{proofofsparseT} and Section \ref{Sec 4}, respectively.
  The proof of Theorem \ref{Extrapolation} is provided in the last section.

To simplify the notations throughout this paper,
we write $X \lesssim Y$ to indicate the existence of a constant
 $C$ such that $X \leq CY$.
 Positive constants may vary across different occurrences. If we write $X \approx Y$,
 then both $X \lesssim Y$ and $Y \lesssim X$ hold.

\section{Preliminaries}\label{preliminaries}

In this section, we first introduce some basic definitions and results in the Dunkl settings. For details we refer to \cite{Dunkl, Roesler3, Roesler-Voit}.

{\bf Dunkl measure.}
For any Euclidean ball $B(x,r)=\{y\in \mathbb{R}^N: \|x-y\|<r\}$ centred at $x\in \mathbb{R}^N$ with radius $r>0$,
the scaling property
$$
 \omega(B(tx, tr))=t^{\mathbf{N}}\omega(B(x,r))
$$
holds. The number $\mathbf{N}$ is the homogeneous dimension associated with the Dunkl setting.

Since
\begin{equation*}
\omega(B(x,r))\approx r^{N}\prod_{\upsilon\in R}(\,|\langle \upsilon,x\rangle|+r\,)^{\kappa(\upsilon)}.
\end{equation*}
The measure $\omega$ satisfies the doubling condition, that is, there is a constant $C_d>0$ such that
\begin{equation*}\label{eqdoubling}
\omega(B(x,2r))\leq C_d\omega(B(x,r)).
\end{equation*}
It implies from the doubling condition that
\begin{equation}\label{equival vollum}
\omega(B(x,r))\approx \omega(B(y,r)), \text{ if } \|x-y\|\approx r.
\end{equation}
Moreover, $\omega$ is also a reverse doubling measure. There exists a constant $C\ge1$ such that,
for every $x\in\mathbb{R}^N$ and for every $r_1\geq r_2>0$,
\begin{equation}\label{growth}
C^{-1}\Big(\frac{r_1}{r_2}\Big)^{N}\leq\frac{{\omega}(B(x,r_1))}{{\omega}(B(x,r_2))}\leq C \Big(\frac{r_1}{r_2}\Big)^{\mathbf N}.
\end{equation}

It is direct to check that $d(x,y)=d(y,x)$ and $d(x,y)\leq d(x,z)+d(z,y)$ for any $x,y,z\in \mathbb{R}^N$.
 The ball defined via the Dunkl metric $d$ is
 $$\mathcal{O}(B(x,r)):=\{y\in \mathbb{R}^N:d(y,x)<r\}=\bigcup_{\sigma \in G}B(\sigma(x),r).
 $$
 Denote by $\lambda\mathcal{O}(B(x,r))=\mathcal{O}(B(x,\lambda r))$ for any $\lambda>0$.

Since $G$ is a finite group, we have
$$
\omega(B(x,r))\leq \omega(\mathcal{O}(B(x,r)))\leq |G|\,{\omega}(B(x,r)).
$$
Combining this with \eqref{equival vollum}, we have
\begin{equation}\label{equival O vollum}
\omega(\mathcal{O}(B(x,r)))\approx \omega(\mathcal{O}(B(y,r))), \text{ if } d(x,y)\approx r.
\end{equation}

\smallskip

{\bf Dunkl operator.}
Given the reflection group $G$ of a root system $R$ and a fixed nonnegative multiplicity function $\kappa$.
$R^+$ is a positive subsystem of $R$ where the elements span a cone in the space of roots.
The Dunkl operators $T_\xi$ introduced in \cite{Dunkl} are defined by the difference operators:
\begin{align*}
T_\xi f(x)
&{=\partial_\xi f(x)+\sum_{\upsilon\in R}\frac{\kappa(\upsilon)}2\langle \upsilon,\xi\rangle
\frac{f(x)-f(\sigma_{\upsilon}(x))}{\langle \upsilon,x\rangle}}\\
&=\partial_\xi f(x)+\sum_{\upsilon\in R^+}\kappa(\upsilon)\langle \upsilon,\xi\rangle\frac{f(x)-f(\sigma_\upsilon(x) )}{\langle \upsilon,x\rangle}.
\end{align*}
$T_\xi$ are the deformations of the directional derivatives $\partial_\xi$.

Set $T_j=T_{e_j}$, where $\{e_1,\dots,e_N\}$ is the canonical basis of $\mathbb{R}^N$. Then, the Dunkl Laplacian ${\Delta}:=\sum_{j=1}^NT_{j}^2$ is the differential-difference operator, which acts on $C^2$ functions by
$$
{\Delta}f(x)
{=\Delta_{\text{eucl}}f(x)
+\sum_{\upsilon\in R}\kappa(\upsilon)\delta_\upsilon f(x)}
=\Delta_{\text{eucl}}f(x)
+2\sum_{\upsilon\in R^+}  \kappa(\upsilon)\delta_\upsilon f(x),
$$
where
$$
\delta_\upsilon f(x)
=\frac{\partial_\upsilon f(x)}{\langle \upsilon,x\rangle}-\frac{f(x)-f(\sigma_\upsilon(x))}{\langle \upsilon,x\rangle^2}
$$
and $\Delta_{\text{eucl}}=\sum_{j=1}^N\partial_{j}^2$ is the classic Laplacian on $\mathbb{R}^N$.

The operator $\Delta$ is essentially self-adjoint on $L^2(\mathbb{R}^N, d\omega)$
 (see for instance \cite[Theorem\;3.1]{AH})
and generates the heat semigroup
$$
H_tf(x)=e^{t{\Delta}}f(x)=\int_{\mathbb{R}^N}
h_t(x,y)f(y)\,   d\omega (y).
$$
Here the heat kernel $h_t(x,y)$ is a $C^\infty$ function in all variables $t>0$, $x,y\in\mathbb{R}^N$.

\smallskip

{\bf Dunkl kernel.} For fixed ${x}\in\mathbb{R}^N$, considering the simultaneous eigenfunction problem
$$
T_\xi f=\langle x,\xi\rangle\,f,\quad\forall\;\xi\in\mathbb{R}^N.
$$
its unique solution $f(y)=E(x,y)$ satisfying $f(0)=1$ is the Dunkl kernel.
The following integral formula was {obtained} by R\"osler \cite{Roesle99}\,{:}
\begin{equation}\label{Rintegral}
{E}(x,y)=\int_{\mathbb{R}^N} e^{\langle\eta, y\rangle} d\mu_{x}(\eta),
\end{equation}
where ${\mu_{x}}$ is a probability measure supported {in the convex hull} $\operatorname{conv}\mathcal{O}(x)$ of the $G$-orbit of $x$. The function ${E}(x,y)$ extends holomorphically to $\mathbb{C}^N\times\mathbb{C}^N$. Please refer to \cite{ADH} for more properties for the Dukl kernel.

\smallskip

{\bf Dunkl transform and Dunkl translation.}
The Dunkl transform is defined by
$$
    \mathcal{F}_{\kappa} f(\xi)=c_{\kappa}^{-1}\int_{\mathbb{R}^N}f(x)E(x,-i\xi) d\omega(x).
$$
Here, $ c_{\kappa}=\int_{\mathbb{R}^N}e^{-\|x\|^2/2}d\omega(x)$ and the function $E(x,y)$ on
$\mathbb{C}^N \times \mathbb{C}^N$
is the Dunkl kernel which generalizes the exponential function $e^{\langle x,y\rangle}$ in the Fourier transform $
    \hat{f}(\xi)=\int_{\mathbb{R}^N}f(x)e^{\langle x,-i\xi\rangle} dx
$.

The Dunkl transform is a topological automorphism of the Schwartz space $\mathcal{S}(\mathbb{R}^N)$. For every $f\in\mathcal{S}(\mathbb{R}^N)$ and actually for every $f\in L^1({\mathbb{R}^N, d\omega})$ such that $\mathcal{F}_{\kappa}f\in L^1({\mathbb{R}^N, d\omega})$, we have
$$
f(x)=\big(\mathcal{F}_{\kappa}\big)^2f(-x),
\quad{\forall\;x\in\mathbb{R}^N}.
$$
Moreover, the Dunkl transform extends to an isometric automorphism of $L^2({\mathbb{R}^N, d\omega})$ (see \cite{dJ, Roesler-Voit}).

 There also exists a Dunkl translation $\tau$ which is defined via Dunkl transform and serves as an analogue to the ordinary translation $\tau_x f(\cdot)=f(\cdot-x)$.

The Dunkl translation $\tau_{x}f$ of a function $f\in\mathcal{S}(\mathbb{R}^N)$ by $x\in\mathbb{R}^N$ is defined by
\begin{equation*}
\tau_{x} f(y)=c_{\kappa}^{-1} \int_{\mathbb{R}^N}{E}(i\xi,x)E(i\xi,y)\mathcal{F}_{\kappa}f(\xi)  \,d\omega(\xi).
\end{equation*}
Notice that each translation $\tau_{x}$ is a continuous linear map of $\mathcal{S}(\mathbb{R}^N)$ into itself, which extends to a contraction on $L^2({\mathbb{R}^N, d\omega})$. The Dunkl translations $\tau_{x}$ and the Dunkl operators $T_\xi$ all commute. For all $x,y\in\mathbb{R}^N$, and
$f,g\in\mathcal{S}(\mathbb{R}^N)$, $\tau_{x}$ also satisfies
\begin{itemize}
\item
$\tau_{x}f(y)=\tau_{y}f(x),$

\smallskip
\item
$\displaystyle \int_{\mathbb{R}^N}\tau_{x}f(y)g(y)\, d\omega(y)
=\int_{\mathbb{R}^N}f(y)\tau_{-x}g(y)  \, d\omega(y)$.
\end{itemize}

The following specific formula was obtained by R\"osler \cite{Roesler2003}
for the Dunkl translations of radial functions $f(x)=\tilde{f}({\|x\|})$\,:
$$
\tau_{x}f(-y)=\int_{\mathbb{R}^N}{\big(\tilde{f}\circ A\big)}(x,y,\eta)\,d\mu_{x}(\eta),{\qquad\forall\;x,y\in\mathbb{R}^N.}
$$
Here
\begin{equation*}
A(x,y,\eta)=\sqrt{{\|}x{\|}^2+{\|}y{\|}^2-2\langle y,\eta\rangle}=\sqrt{{\|}x{\|}^2-{\|}\eta{\|}^2+{\|}y-\eta{\|}^2}
\end{equation*}
and $\mu_{x}$ is the probability measure occurring in \eqref{Rintegral}, which is supported in $\operatorname{conv}\mathcal{O}(x)$.

\smallskip

{\bf Dyadic system in Dunkl setting.}
Since $d(x,y)=0 \Leftrightarrow x=y$ does not hold for Dunkl metric, $(\mathbb{R}^N, d, d\omega)$ is not a space of homogeneous type introduced in 1970s by Coifman and Weiss (see \cite{CoifWeis1971}). However, Bui in \cite{Bui} pointed out that $(\mathbb{R}^N / G, d, d \mu)$ is a space of homogeneous type, where $\mu(E)=\omega(\cup_{\mathcal{O}(x) \in E} \mathcal{O}(x))$. Therefore, there exists a family $\mathscr{D}^d:=\left\{Q_\alpha^{k, d}: k \in \mathbb{Z},~\alpha \in I_k^d\right\}$ (see \cite{HK2012}). Then, the set $\mathscr{D}^d$ is called the system of dyadic cubes in $(\mathbb{R}^N, d, d \omega)$.

 Note that the elements of the space $(\mathbb{R}^N / G, d, d \mu)$ are orbits $\mathcal{O}(\cdot)$.
Since $d(x,\mathcal{O}(x))=0$ and $d(\mathcal{O}(y),\mathcal{O}(x))=d(y,x)$ for any $x,y\in \mathbb{R}^N$, we can write:
 \begin{lemma}\label{dyadicsystemfit}
Fix $\delta \in(0, 1 )$. Then there exists a family of sets $\left\{Q_\alpha^{k,d}: k \in \mathbb{Z},~\alpha \in I_k^d\right\}$ and a set of points $\left\{x_{Q_\alpha^{k,d}}: k \in \mathbb{Z},~\alpha \in I_k^d\right\}$ and satisfying
\begin{enumerate}
  \item $d(x_{Q_\alpha^{k,d}}, x_{Q_\beta^{k,d}}) \geq \delta^k$ for all $k \in \mathbb{Z}$ and $\alpha,~\beta \in I_k^d$ with $\alpha \neq \beta$;
  \item $\min_{\alpha \in I_k^d} d(x, x_{Q_\alpha^{k,d}}) \leq \delta^k$ for all $x \in \mathbb{R}^N $;
 \item for any $k \in \mathbb{Z}, \bigcup_{\alpha \in I_k^d} Q_\alpha^{k,d}=\mathbb{R}^N$ and $\{Q_\alpha^{k,d}: \alpha \in I_k^d\}$ is disjoint;
  \item if $k, l \in \mathbb{Z}$ and $k \geq l$, then either $Q_\alpha^{k,d} \subset Q_\beta^{l,d}$ or $Q_\alpha^{k,d} \cap Q_\beta^{l,d}=\emptyset$ for every $\alpha \in I_k^d$ and $\beta \in I_{l}^d$;
  \item \label{QsubB}
  for any $k \in \mathbb{Z}$ and $\alpha \in I_k^d, \mathcal{O}(B(x_{Q_\alpha^{k,d}}, \delta^k / 6)) \subset Q_\alpha^{k,d} \subset \mathcal{O}(B(x_{Q_\alpha^{k,d}}, 2 \delta^k))$.
\end{enumerate}
\end{lemma}

For each $k \in \mathbb{Z}$, we denote $\mathscr{D}_k^d=\{Q_\alpha^{k, d}: \alpha \in I_k^d\}$. For each $k \in \mathbb{Z}$ and $\alpha \in I_k^d$, we denote $\ell(Q_\alpha^{k, d})=\delta^k$ and
$\lambda B(Q_\alpha^k):=\mathcal{O}(B(x_{Q_\alpha^{k, d}}, 2\lambda \delta^k)), \lambda>0$.

Now, let us recall the definition of adjacent systems of dyadic cubes.

\begin{definition} (\cite{DGKLWY, HK2012})  \label{defi adjacent dyadic cube}
On $(\mathbb{R}^N , d, d \omega)$, a finite collection $\{\mathscr{D}^{d,t}: t=1,2, \ldots, C_{\mathbf{N}}\}$ of the dyadic families is called a collection of adjacent systems of dyadic cubes with parameters $\delta \in(0,1)$ and $C_0\geq1$ if it has the following properties: individually, each $\mathscr{D}^{d,t}$ is a system of dyadic cubes with parameters $\delta\in(0,1/24)$; collectively, for each ball $\mathcal{O}(B(x, r)) \subseteq  \mathbb{R}^N $ with $\delta^{k+3}<r \leq \delta^{k+2}, k \in \mathbb{Z}$, there exist $t \in\{1,2, \ldots, C_{\mathbf{N}}\}$ and $Q^{k,d,t} \in \mathscr{D}^{d,t}$ with $\ell(Q^{k,d,t})=\delta^k$ and center point $ x_{Q^{k,d,t}}$ such that $d (x, x_{Q^{k,d,t}} )<2 \delta^k$ and
\begin{align}\label{BsubsetQ}
\mathcal{O}(B(x, r)) \subseteq Q^{k,d,t} \subseteq \mathcal{O}(B(x, C_0r)).
\end{align}
\end{definition}
Here, $C_0$ only depends on $\delta$.
We refer to \cite{HK2012} for the constructions of the adjacent systems of dyadic cubes. From the analysis in \cite{KLPW}, one can see that the number of the adjacent systems of dyadic cubes $C_{\mathbf{N}}$ in Definition \ref{defi adjacent dyadic cube} is finite and depends on $\mathbf{N}$, which can be omitted in the following calculation.
\begin{remark}
\label{BtoQ}
By \eqref{BsubsetQ}, we can see that the definitions of $[u]_{A^d_{p}}$, $\|b\|_{\mathrm{BMO}^d_u}$ and $M_d $ defined via the Dunkl balls $\mathcal{O}(B)$ are equivalent to those defined in terms of the Dunkl cubes $Q^d\in \mathscr{D}^d$.

Moreover, for $k\geq 0$ and for $Q^{k+1,d} \subsetneqq Q^{k,d}$ where $Q^{k+1,d}$ denotes the largest dyadic subcube in $\mathscr{D}(Q^{k,d})$. Notice that $d(x_{Q^{k,d}}, x_{Q^{k+1,d}}) \approx 2\ell(Q^{k,d})$, by using \eqref{equival O vollum} and \eqref{QsubB} in Lemma \ref{dyadicsystemfit}, we have the doubling measure condition
\begin{equation}
\label{doublingQ}
\begin{aligned}
\omega(Q^{k,d})\leq \tilde{C}_{d}\omega(Q^{k+1,d}).
 \end{aligned}
 \end{equation}
Here, $\tilde{C}_{d}$ depends on $C_{d}$.
\end{remark}

\begin{definition} (\cite{DGKLWY, HK2012})
 Given $0<\theta<1$, a collection $\mathscr{S}^d \subset \mathscr{D}^d$ of dyadic cubes is said to be $\theta$-sparse family provided that for every $Q^d \in \mathscr{S}^d$, if there is a measurable subset $E(Q^d) \subset Q^d$ such that $\omega(E(Q^d)) \geq \theta \omega(Q^d)$ and the sets $\{E(Q^d)\}_{Q^d \in \mathscr{S}^d}$ have only finite overlap. That is, there exists a constant $C \geq 1$ such that $\sum_{Q^d} \mathbf{1}_{E(Q^d)}(x) \leq C$ for all $x \in \mathbb{R}^{N}$.
\end{definition}

\smallskip

{\bf Dunkl-Muckenhoupt weight function.}
 On $(\mathbb{R}^N, d, d \omega)$, a function $u\geq0$ is said to be an $A^d_1$ weight if
$$
M_du(x) \leq C_{A^d_1}u(x)
$$
for almost every $x \in (\mathbb{R}^N,d,d\omega)$.
We denote by
$$
[u]_{A^d_1}:=\operatorname*{ess sup}_{x \in \mathbb{R}^N} \frac{M_d u(x)}{u(x)}.
$$

Now, we have the following useful properties for the weight function in the Dunkl setting.
\begin{proposition}
 There exist the following properties for the weight function:
\begin{enumerate}
\item
 Let $1<p<\infty$. If $u \in A^d_p$ then
\begin{align}\label{wp}
\left(\frac{\omega(E)}{\omega(Q^d)}\right)^p \leq[u]_{A^d_p} \frac{u(E)}{u(Q^d)}
\end{align}
for any measurable subset $E=\{\mathcal{O}(x)\}_{x} \subset Q^d \in \mathscr{D}^d$.

\item For $1 \leq p<q<\infty$, we have
\begin{align}\label{wpq}
A^d_p \subset A^d_q \text{ with } [u]_{A^d_q}\leq  [u]_{A^d_p}.
\end{align}

\item (Reverse H\"{o}lder Inequality) On $(\mathbb{R}^N, d, d \omega)$, let $u$ be a radial weight, there exist $0<C,\gamma <\infty $ such that for every cube $Q^d$ in $\mathbb{R}^N$, we have
\begin{align}\label{reverseHolder}
\left(\frac{1}{\omega(Q^d)}\int_{Q^d}u(x)^{1+\gamma}d\omega(x)\right)^{1/(1+\gamma)}
\leq \frac{C}{\omega(Q^d)}\int_{Q^d}u(x)d\omega(x).
\end{align}
\end{enumerate}
\end{proposition}
\begin{proof}
By the inequality on Page 505 of \cite{GrafakosCla2014} used to prove $(8)$ in Proposition 7.1.5, let $f=\mathbf{1}_{E}$, we can prove \eqref{wp}. Similar to the proof of \cite[Proposition 7.1.5]{GrafakosCla2014}, by Remark \ref{BtoQ}, we can directly obtain \eqref{wpq}.

 As for \eqref{reverseHolder}, note that there exists a dyadic system $\mathscr{D}^d$ on $(\mathbb{R}^N, d, d \omega)$. For any
 Dunkl cube $Q^d\in\mathscr{D}^d$ centered at $x$. Since $d\omega$ is a doubling measure with the constant $C_d$, from the proof of \cite[Corollary 7.2.4]{GrafakosCla2014}, we only need to prove the weight $u$ satisfies the Lebesgue differential theorem. In fact, if we apply the Calder\'{o}n-Zygmund decomposition to the function $u$ on $Q^d$ at height $\alpha$. We can decompose $Q^d$ into cubes $\{Q^d_j\}_j$ which satisfies
$$\text{For almost all } x\in Q^{d}\backslash(\cup_j Q^{d}_j), \text{we have } u(x) \leq \alpha,$$
which follows from the Lebesgue differential theorem.

 First, since $Q^d \in \mathscr{D}^d$, we have $(1/12)B(Q^d)\subset Q^d \subset B(Q^d)$ where $B(Q^d)=\mathcal{O}(B(x,2\ell(Q^d)))$. Then,
 denote by
 $$\frac{1}{\omega(Q^d)}\int_{Q^d}u(y)d\omega(y)=:U(x)$$
 we have
 $$
\frac{1}{\omega(B(Q^d))}\int_{1/12B(Q^d)}u(y)d\omega(y)
\leq
U(x)\leq \frac{1}{\omega(1/12B(Q^d))}\int_{B(Q^d)}u(y)d\omega(y).
 $$
We can deduce from $B(x,1/6\ell(Q^d))\subset 1/12 B(Q^d)$ that
\begin{align*}
&\frac{1}{|G|\omega(B(x,2\ell(Q^d)))}\int_{B(x,1/6\ell(Q^d))}u(y)d\omega(y)\\
\leq &
U(x)
\leq  \sum_{\sigma \in G}\frac{1}{\omega(B(\sigma(x),1/6\ell(Q^d)))}\int_{B(\sigma(x),2\ell(Q^d))}u(y)d\omega(y).
\end{align*}
Let $\ell(Q^d) \to 0$, by the Lebesgue differential theorem on $(\mathbb{R}^N, \|\cdot\|, d \omega)$, we have
$$
\frac{1}{|G|}u(x)
\leq
U(x)
\leq  \sum_{\sigma \in G}u(\sigma(x)).
$$
 Since $u$ is a radial weight, we obtain
$$
U(x)\approx u(x)
$$
as $\ell(Q^d) \to 0$. This completes the proof.
\end{proof}

{\bf Maximal function.}
The Hardy-Littlewood maximal function $M$ in the Dunkl setting is defined as
$$
M f(x)=\sup _{B\ni x, B\subset \mathbb{R}^N} \frac{1}{\omega(B)} \int_B|f(y)| \,   d \omega(y).
$$
Note that $(\mathbb{R}^N, \|\cdot\|, d\omega)$ is a space of homogeneous type.
Then, $M$ is $L^p$ bounded for $1<p<\infty$ and is of weak $(1,1)$ type bounded (see \cite{St2}).

Suppose that $T$ is a Dunkl-Calder\'on-Zygmund operator, the maximal operator of $T$ is defined by
\begin{align*}
T^{*}f(x)= \sup_{\varepsilon>0}|T_{\varepsilon}f(x)|
\end{align*}
where
\begin{align*}
T_{\varepsilon}f(x)=
\int_{d(x,y)>\varepsilon}K(x,y)f(y)d\omega(y).
\end{align*}
The authors of \cite{TanHanLi2023} proved that $T^{*}$ is bounded on $L^p(\mathbb{R}^N, d\omega)$ for $1<p<\infty$ and is of weak type $(1,1)$ (see \cite[Corollary 3.2]{TanHanLi2023}).

For $u \in A^d_{p}$, $1<p<\infty$, we define the weighted dyadic Hardy-Littlewood maximal function $M^d_{u}$ by
$$
M^d_u f (x)= \sup_{Q^d \in \mathscr{D}^d}  \frac{1}{u(Q^d)} \int_{Q^d}|f(y)|u(y) d \omega(y) \mathbf{1}_{Q^d}(x).
$$

Then, we have the following result.
 \begin{lemma}\label{Muweight}
 For $1<p<\infty$, $u\in A^d_p$, then $M^d_u$ is bounded on $L^p(\mathbb{R}^N,d, ud\omega)$.
 \end{lemma}
\begin{proof}
On $(\mathbb{R}^N,d,d\omega)$, for any weight function $u\in A^d_p$, by using \eqref{wp} and  \eqref{doublingQ}, we can see for any $Q^{k+1,d} \subset Q^{k,d}$ with $Q^{k+1,d}\in \mathscr{D}^d(Q^{k,d})$, $k\geq 0$,
\begin{equation}\label{boublingforu}
\begin{aligned}
u(Q^{k,d})\leq &[u]_{A_p^d} \left(\frac{\omega(Q^{k+1,d})}{\omega(Q^{k,d})}\right)^p u(Q^{k+1,d})
\leq  C[u]_{A_p^d} (\tilde{C}_{d})^{-p} u(Q^{k,d}) =: C_{u} u(Q^{k,d}).
\end{aligned}
\end{equation}
Then, we use the Calder\'on-Zygmund decomposition to the function $|f|$ on $\mathbb{R}^N$ at the height $\alpha/C_{u}$. We can get a set of disjoint dyadic cubes $\{Q^{d}\} \subset \mathscr{D}^d$ such that for each $Q^d$, we have
$$
 \frac{1}{u(Q^d)} \int_{Q^d}|f(y)|u(y) d \omega(y) > \frac{\alpha}{C_u}.
$$
Notice that
\begin{equation}\label{equQd}
\begin{aligned}
&\int_{Q^d}u(x) d\omega(x) \left(\frac{1}{u(Q^d)}\int_{Q^d}|f(x)|u(x)d\omega(x)\right)^p\\
\leq & u(Q^d) \frac{1}{u(Q^d)^p}\left(\int_{Q^d}|f(x)|^pu(x)d\omega(x)\right)
\left(\int_{Q^d}u(x)d\omega(x)\right)^{p-1}\\
= &\int_{Q^d}|f(x)|^pu(x)d\omega(x).
\end{aligned}
\end{equation}
Keeping \eqref{boublingforu} and \eqref{equQd} in mind and following the standard argument of the proof for \cite[Theorem 1.4.3]{LuDingYan}, the lemma can be proved.
\end{proof}
\medskip

\section{Proof of Theorem \ref{sparseforT}}\label{proofofsparseT}
Let $$\tilde{C}_{0} := 4(\lfloor 2C_0 \rfloor + 1),$$ where $C_0$ is the constant appearing in \eqref{BsubsetQ}. Here, $\lfloor t \rfloor$ denotes the greatest integer less than or equal to $t \in \mathbb R$.
 In the Dunkl setting, let us define the grand maximal truncated operator on $\mathbb{R}^N$ by
\begin{align*}
M_{T} f(x):=
\sup_{\substack{\mathcal{O}(B) \supset \mathcal{O}(x)\\ \mathcal{O}(B) \subset \mathbb{R}^N}}  \operatorname*{ess\,sup}_{\mathcal{O}(\xi) \subset \mathcal{O}(B)}
\sup_{\sigma \in G}
|T(f \mathbf{1}_{\mathbb{R}^N \backslash \tilde{C}_{0}  \mathcal{O}(B)})(\sigma(\xi))|.
\end{align*}
For any fixed $\mathcal{O}(B_0) \subset \mathbb{R}^N$, the locally grand maximal truncated operator on $\mathbb{R}^N$ by
\begin{align*}
M_{T,\mathcal{O}(B_0)} f(x):=
\sup_{\substack{\mathcal{O}(B) \supset \mathcal{O}(x)\\ \mathcal{O}(B) \subset \mathcal{O}(B_0)}} \operatorname*{ess\,sup}_{\mathcal{O}(\xi) \subset \mathcal{O}(B)}
\sup_{\sigma \in G}
|T(f \mathbf{1}_{\tilde{C}_{0} \mathcal{O}(B_0) \backslash \tilde{C}_{0}  \mathcal{O}(B)})(\sigma(\xi))|.
\end{align*}

Note that for any Dunkl ball $\mathcal{O}(B) \subset \mathbb{R}^N$ and any point $x \in \mathcal{O}(B)$, $x \in \mathcal{O}(B) \Leftrightarrow \mathcal{O}(x) \subset \mathcal{O}(B)$.
Observe that for any $\mathcal{O}(B) \subset \mathbb{R}^N$,
\begin{equation}\label{MTB0<MT}
\begin{aligned}
M_{T,\mathcal{O}(B_0)} f(x)\leq &
\sup_{\substack{\mathcal{O}(B) \supset \mathcal{O}(x)\\ \mathcal{O}(B) \subset \mathbb{R}^N}}
 \operatorname*{ess\,sup}_{\mathcal{O}(\xi) \subset \mathcal{O}(B)}
\sup_{\sigma \in G}
|T(f \mathbf{1}_{\tilde{C}_{0} \mathcal{O}(B_0)}\mathbf{1}_{\mathbb{R}^N \backslash \tilde{C}_{0} \mathcal{O}(B)})(\sigma(\xi))|\\
=&M_{T} (f\mathbf{1}_{\tilde{C}_{0}  \mathcal{O}(B_0)})(x).
\end{aligned}
\end{equation}
Next, we have the following lemma.
\begin{lemma}\label{lemmaMT}
For a.e. $x \in \mathcal{O}(B_0)$, we have
\begin{align}\label{eqMB0}
 |T(f \mathbf{1}_{\tilde{C}_{0}  \mathcal{O}(B_0) })(x)|
\lesssim  |G|\|T\|_{L^1 \rightarrow L^{1, \infty}} \sum_{\sigma \in G}|f(\sigma(x))|
       + M_{T, \mathcal{O}(B_0)} f(x).
\end{align}
and
\begin{align}\label{eq35}
M_{T} f(x)
\lesssim \sum_{\sigma \in G}|Mf(\sigma(x))| + \sum_{\sigma \in G}|T^{*}f(\sigma(x))|.
\end{align}
\end{lemma}

\begin{proof}[Proof of Lemma \ref{lemmaMT}]
First, we prove \eqref{eqMB0} by applying the ideas in \cite{Lerner2016}.
Let $x$ belong to the interior of the Dunkl ball $\mathcal{O}(B_0)=\{y\in \mathbb{R}^N: d(x_0, y)<r_0\}$. Then for every $\sigma \in G$, we have $\sigma(x)$ belongs to the interior of $\mathcal{O}(B_0)$. By the continuity of $T(f \mathbf{1}_{\tilde{C}_{0} \mathcal{O}(B_0)})(\cdot)$, for each $\sigma \in G$ and for any $\varepsilon >0$, we define the set
$$
E_{s}(\sigma(x)):=\{y \in B(\sigma(x),s):
|T(f \mathbf{1}_{\tilde{C}_{0} \mathcal{O}(B_0)})(y)-T(f \mathbf{1}_{\tilde{C}_{0} \mathcal{O}(B_0)})(\sigma(x))|<\varepsilon\},
$$
which satisfies
$$
\lim_{s\rightarrow 0}\frac{\omega(E_{s}(\sigma(x)))}{\omega( B(\sigma(x),s))}=1.
$$
Since $B(\sigma(x),s)$ is a Euclidean ball for each $\sigma \in G$.
Let $s$ small enough such that
$B(\sigma(x),s) \subset \mathcal{O}(B_0)$.
Therefore, it is direct to get $$\mathcal{O}(B(x,s))= \bigcup_{\sigma}B(\sigma(x),s) \subset \mathcal{O}(B_0).$$
Note that $\mathcal{O}(x) \subset \mathcal{O}(B(x,s))$. By the calculation of the first inequality on the page 344 in \cite{Lerner2016}.
For almost everywhere $y\in E_{s}(\sigma(x))$, we have
\begin{align*}
 &|T(f \mathbf{1}_{\tilde{C}_{0}  \mathcal{O}(B_0) })(\sigma(x))|\\
\leq & |T(f \mathbf{1}_{\tilde{C}_{0}  \mathcal{O}(B_0) })(\sigma(x))-
T(f \mathbf{1}_{\tilde{C}_{0}  \mathcal{O}(B_0) })(y)|+
|T(f \mathbf{1}_{\tilde{C}_{0}  \mathcal{O}(B_0)\backslash \tilde{C}_{0}\mathcal{O}(B(x,s)) })(y)|
+ |T(f \mathbf{1}_{\tilde{C}_{0} \mathcal{O}(B(x,s)) })(y)|\\
\leq & \varepsilon + M_{T, \mathcal{O}(B_0)} f(x) + \|T\|_{L^1 \rightarrow L^{1, \infty}} \frac{1}{\omega(E_{s}(\sigma(x)))}
\frac{|G|\omega(B(\sigma(x),s))}{\omega(\mathcal{O}(B(x,s)))} \int_{\tilde{C}_{0} \mathcal{O}(B(x,s))} |f(y)|d\omega (y)
\\
\leq & \varepsilon + M_{T, \mathcal{O}(B_0)} f(x) + |G|\|T\|_{L^1 \rightarrow L^{1, \infty}} \frac{\omega(B(\sigma(x),s))}{\omega(E_{s}(\sigma(x)))}
\sum_{\sigma \in G}\frac{1}{\omega(B(\sigma(x),s))} \int_{B(\sigma(x),\tilde{C}_{0}s)} |f(y)|d\omega (y).
\end{align*}
Let $\varepsilon$ and $s$ tend to $0$, by the Lebesgue differential theorem on $(\mathbb{R}^N, \|\cdot\|, d \omega)$, we can obtain
\begin{align*}
 |T(f \mathbf{1}_{\tilde{C}_{0}  \mathcal{O}(B_0) })(x)|
 \leq & \sup_{\sigma \in G}|T(f \mathbf{1}_{\tilde{C}_{0}  \mathcal{O}(B_0) })(\sigma(x))|\\
\lesssim &  |G|\|T\|_{L^1 \rightarrow L^{1, \infty}} \sum_{\sigma \in G}|f(\sigma(x))|
       +M_{T, \mathcal{O}(B_0)} f(x).
\end{align*}

Next, let us prove \eqref{eq35}.
For any $x \in \mathbb{R}^N$ and any $\mathcal{O}(B)=\{y\in \mathbb{R}^N: d(x_c, y)<r\}$ containing $x$, we have $\mathcal{O}(x) \subset \mathcal{O}(B)$.
For any $\mathcal{O}(\xi) \subset \mathcal{O}(B)$ and any $\sigma(\xi) \in \mathcal{O}(\xi) $, we assume that there exists $\sigma' \in G$ such that
$$\|\sigma(\xi) -\sigma'(x) \|=\min_{\tilde{\sigma} \in G }\|\sigma(\xi) -\tilde{\sigma}(x)\|=d(\sigma(\xi), x)=d(\xi,x).$$
Note that $\mathcal{O}(B) \subset \mathcal{O}(B(x, 8 r))$. We have
\begin{equation}\label{splitTinlemma}
\begin{aligned}
&|T(f \mathbf{1}_{\mathbb{R}^N \backslash \tilde{C}_{0} \mathcal{O}(B)})(\sigma(\xi))|\\
\leq & |T(f \mathbf{1}_{\mathbb{R}^N \backslash \tilde{C}_{0} \mathcal{O}(B(x, 8 r))})(\sigma(\xi))-
T(f \mathbf{1}_{\mathbb{R}^N \backslash \tilde{C}_{0} \mathcal{O}(B(x, 8 r))})(\sigma'(x))|\\
& +|T(f \mathbf{1}_{\mathbb{R}^N \backslash \tilde{C}_{0} \mathcal{O}(B(x, 8 r))})(\sigma'(x))|
     +|T(f \mathbf{1}_{\tilde{C}_{0} \mathcal{O}(B(x, 8 r))\backslash \tilde{C}_{0} \mathcal{O}(B)})(\sigma(\xi))|\\
=: & \mathcal{T}_1(\xi,x) + \mathcal{T}_2(\xi,x) + \mathcal{T}_3(\xi,x).
\end{aligned}
\end{equation}
For any $y\in \mathbb{R}^N \backslash \tilde{C}_{0} \mathcal{O}(B(x, 8 r))$, $\|\sigma(\xi)-\sigma'(x)\|=d(\xi,x)< 2r< d(x,y)/2=d(\sigma'(x),y)/2$. Therefore, we can use \eqref{smooth x3CZ} and the inequality on right hand side of \eqref{growth} and then get
\begin{align*}
 \mathcal{T}_1(\xi,x)
\leq &
       \int_{\mathbb{R}^N \backslash \tilde{C}_{0} \mathcal{O}(B(x, 8 r))}
       |K(\sigma(\xi),y)-K(\sigma'(x),y)||f(y)|d\omega(y)\\\nonumber
\leq & \int_{\mathbb{R}^N \backslash \tilde{C}_{0} \mathcal{O}(B(x, 8 r))}
       \Big(\frac{\|\sigma(\xi)-\sigma'(x)\|}{d(\sigma'(x),y)}\Big)^\varepsilon
    \frac1{\omega(B(\sigma'(x),d(\sigma'(x),y)))}
    |f(y)|d\omega(y)
    \\\nonumber
\leq &  \sum_{i\geq 1}\int_{\tilde{C}_{0}2^{i+3}r<d(x,y)\leq \tilde{C}_{0}2^{i+4}r }
       \frac{(2r)^\varepsilon(\tilde{C}_{0}2^{i+4}r)^{\mathbf{N}}}{d(\sigma'(x),y)^{\varepsilon+\mathbf{N}}}
    \frac{1}{\omega(B(\sigma'(x),\tilde{C}_{0}2^{i+4}r))}
    |f(y)|d\omega(y)
    \\\nonumber
\lesssim & \sum_{i\geq 1} 2^{-i\varepsilon}
     \frac{1}{\omega(\mathcal{O}(B(\sigma'(x),\tilde{C}_{0}2^{i+4}r)))}
   \int_{d(x,y)\leq \tilde{C}_{0}2^{i+4}r }
    |f(y)|d\omega(y)
    \\\nonumber
\lesssim & \sum_{\sigma \in G}
     Mf(\sigma(x)).
\end{align*}
Since $d(x,y)=d(\sigma'(x),y)$ for $\sigma' \in G$, we have that
\begin{align*}
 \mathcal{T}_2(\xi,x)
= &
       \left|\int_{\mathbb{R}^N \backslash \tilde{C}_{0} \mathcal{O}(B(x, 8 r))}
       K(\sigma'(x),y)f(y)d\omega(y)\right|\\
= &
       \left|\int_{d(\sigma'(x),y)>8\tilde{C}_{0}r}
       K(\sigma'(x),y)f(y)d\omega(y)\right|\\
\leq & T^{*}f(\sigma'(x))
\leq \sum_{\sigma \in G}T^{*}f(\sigma(x)).
\end{align*}
 For any $x,\xi \in \mathcal{O}(B)$ and any $y \in \tilde{C}_{0} \mathcal{O}(B(x, 8 r))\backslash \tilde{C}_{0} \mathcal{O}(B)$, we have
$(8\tilde{C}_{0}+2)r > d(\sigma(\xi),y)=d(\xi,y)> (\tilde{C}_{0}-2)r$ and $\mathcal{O}(B(x,(\tilde{C}_{0}-4)r)) \subset \mathcal{O} (B(\sigma(\xi),(\tilde{C}_{0}-2) r))$. By using
\eqref{si} and the inequality on the left hand side of \eqref{growth}, we have
\begin{align*}
\mathcal{T}_3(\xi,x)
\leq &
 \int_{\tilde{C}_{0} \mathcal{O}(B(x, 8 r))\backslash \tilde{C}_{0} \mathcal{O}(B)}
 |K(\sigma(\xi),y)||f(y)|d\omega(y)\\
\leq &  \int_{\tilde{C}_{0} \mathcal{O}(B(x, 8 r))\backslash \tilde{C}_{0} \mathcal{O}(B)}
\frac1{\omega(B(\sigma(\xi),d(\sigma(\xi),y)))}
     |f(y)|d\omega(y)\\
\leq &  \int_{\tilde{C}_{0} \mathcal{O}(B(x, 8 r))\backslash \tilde{C}_{0} \mathcal{O}(B)}
\frac1{\omega(B(\sigma(\xi),(\tilde{C}_{0}-2) r))}
    \Big(\frac{(\tilde{C}_{0}-2) r}{d(\sigma(\xi),y)}\Big)^N |f(y)|d\omega(y)\\
\lesssim &  \frac{|G|}{\omega(\mathcal{O}(B(x,(\tilde{C}_{0}-4) r)))}
      \int_{\tilde{C}_{0} \mathcal{O}(B(x, 8 r))}
        |f(y)|d\omega(y)\\
\lesssim &\sum_{\sigma \in G} Mf(\sigma(x)).
\end{align*}

Combining \eqref{splitTinlemma} with the estimates for $\mathcal{T}_1(\xi,x)-\mathcal{T}_3(\xi,x)$, by the definition of $M_Tf$, we have
\begin{align*}
M_{T} f(x)
\lesssim   \sum_{\sigma \in G}Mf(\sigma(x)) + \sum_{\sigma \in G}T^{*}f(\sigma(x)).
\end{align*}

Therefore, this lemma is proved.
\end{proof}

With this key lemma in hand, we are now ready to establish sparse domination for $Tf$ and $[b,T]f$.
Our proof follows the ideas in \cite{DGKLWY} and \cite{Lerner2016}.

\begin{proof}[Proof of \eqref{sparseT} in Theorem \ref{sparseforT}]
For the function $f$ with $\operatorname{supp} f \subset B:= B(x_c, r)$ which is contained in the Dunkl ball $\mathcal{O}(B)$. For any non-negative integer $i$, denote by $2^i\mathcal{O}(B)=\mathcal{O}(B(x_c, 2^ir))$ and
$$
\mathcal{R}_i(B):=2^{i+1}\mathcal{O}(B)\backslash 2^i\mathcal{O}(B)=
\left\{y\in\mathbb{R}^N:2^ir \leq d(x_c,y)<2^{i+1}r\right\}.
$$
 Then, for each $i,$ we can find at most $\mathcal{I}_i$ Dunkl balls $\mathcal{O}(B_{i,t}):=\mathcal{O}(B(x_{i,t},2^{i+2} \tilde{C}_{0}^{-1} r))$ with the centre $x_{i,t}\in \mathcal{R}_i(B)$ such that
$$\mathcal{R}_i(B) \subset \bigcup^{\mathcal{I}_i}_{t=1}\mathcal{O}(B_{i,t}).$$
Note that $\sup_{i}\mathcal{I}_i$ is a finite constant depends on $\mathbf{N}$ and $C_{0}$. In fact,
$x_{i,t}\in \mathcal{R}_i(B)$ implies that $d(x_c,x_{i,t})\approx 2^{i+1}r$. Combining with \eqref{equival O vollum}, we have $\omega(\mathcal{O}(B(x_{c},2^{i+1}r)))\approx \omega(\mathcal{O}(B(x_{i,t},2^{i+1}r)))$.
Then
\begin{align*}
\frac{\omega(\mathcal{R}_i(B))}{\omega(\mathcal{O}(B(x_{i,t},2^{i+2}\tilde{C}_{0}^{-1}r)))}
\approx \frac{\omega(\mathcal{O}(B(x_{i,t},2^{i+1}r)))}{\omega(\mathcal{O}(B(x_{i,t},2^{i+2}\tilde{C}_{0}^{-1}r)))}
\leq \left(\frac{\tilde{C}_{0}}{2}\right)^{\mathbf{N}}.
\end{align*}
Then, we can decompose $\mathbb{R}^N$ according to $\mathcal{O}(B)$ into
$$\mathbb{R}^N
=\bigcup^{\infty}_{i=0}\mathcal{R}_i(B) = \bigcup^{\infty}_{i=0}\bigcup^{\mathcal{I}_i}_{t=1}\mathcal{O}(B_{i,t}).$$

Note that for any fixed $\mathcal{O}(B_0) \in\mathbb{R}^N$, by Definition \ref{defi adjacent dyadic cube} and \eqref{BsubsetQ}, there exist finite dyadic cubes $Q^{d,0} \in \mathscr{D}^{d}$ such that
$
\mathcal{O}(B_0) \subseteq Q^{d,0} \subseteq C_0 \mathcal{O}(B_0).
$
Moreover, from \eqref{QsubB} in Lemma \ref{dyadicsystemfit}, we have $1/12 B(Q^{d,0}) \subset Q^{d,0} \subset B(Q^{d,0})$.
We claim that there exists a $1/2$-sparse family $\mathscr{P}^{d}\subset \mathscr{D}^{d}(Q^{d,0})$ such that
\begin{align}\label{irritatesparse}
 |T(f \mathbf{1}_{\tilde{C}_{0}B(Q^{d,0})})(x)| \mathbf{1}_{Q^{d,0}}(x)
\lesssim \sum_{Q^d \in \mathscr{P}^{d}}
|f|_{\tilde{C}_{0}B(Q^{d})}\mathbf{1}_{ Q^{d}}(x).
\end{align}
Applying this claim to each $\mathcal{O}(B_{i,t}) \in \{\mathcal{O}(B_{i,t})\}_{i,t}$, we get a set of $\{Q^{d}_{i,t}\}_{i,t} \in \mathscr{D}^{d}$ such that
$
\mathcal{O}(B_{i,t}) \subseteq Q^{d}_{i,t} \subseteq C_0 \mathcal{O}(B_{i,t})
$
and $1/12 B(Q^{d}_{i,t}) \subset Q^{d}_{i,t} \subset B(Q^{d}_{i,t})$ and a corresponding $1/2$-sparse family $\mathscr{P}^d_{i,t} \subset \mathscr{D}^d(Q^d_{i,t})$ such that \eqref{irritatesparse} holds:
$$ |T(f \mathbf{1}_{\tilde{C}_{0}B(Q^{d}_{i,t})})(x)| \mathbf{1}_{Q^{d}_{i,t}}(x)
\lesssim \sum_{Q^d \in \mathscr{P}^{d}}
|f|_{\tilde{C}_{0}B(Q^{d})}\mathbf{1}_{ Q^{d}}(x).$$
It is easy to check that for any $\mathcal{O}(B_{i,t})$,
$$
C_{0}\mathcal{O}(B_{i,t}) \cap \mathcal{R}_{i+\lfloor \log \tilde{C}_{0}\rfloor }(B)= \emptyset \quad\text{and}\quad
C_0\mathcal{O}(B_{i,t}) \cap \mathcal{R}_{i-\lfloor \log \tilde{C}_{0}\rfloor }(B)= \emptyset.
$$
Then, we conclude that $C_0\mathcal{O}(B_{i,t})$ overlaps finite times. Let $\mathscr{F}^d:=\bigcup_{i,t}\mathscr{P}^d_{i,t}$,
we have
\begin{align*}
 |T(f )(x)|
\lesssim \sum_{Q^d \in \mathscr{F}^d}
|f|_{\tilde{C}_{0}B(Q^{d})}\mathbf{1}_{ Q^{d}}(x).
\end{align*}
By using \eqref{BsubsetQ} again, for each Dunkl ball $B(Q^{d})$ with $Q^d \in \mathscr{F}^d$, we can find another dyadic cube $Q^d_{B(Q^{d})}$ such that $\tilde{C}_{0}B(Q^{d}) \subseteq Q^d_{B(Q^{d})} \subseteq C_0 \tilde{C}_{0}B(Q^{d})$.
Then, from the inequality in the right hand side of \eqref{growth}, we obatin
$$|f|_{\tilde{C}_{0}B(Q^{d})} \lesssim |f|_{Q^d_{B(Q^{d})}}.$$
Then, we have
\begin{align*}
 |T(f )(x)|
\lesssim \sum_{Q^d \in \mathscr{F}^d}
|f|_{Q^d_{B(Q^{d})}}\mathbf{1}_{ Q^d_{B(Q^{d})}}(x).
\end{align*}
Let $\mathscr{S}^d:=\{Q^d_{B(Q^{d})}: Q^d \in \mathscr{F}^d\}$, we get the desired result
\begin{align*}
 |Tf (x)|
\lesssim \sum_{Q^d \in \mathscr{S}^d}
|f|_{Q^{d}}\mathbf{1}_{ Q^{d}}(x),
\end{align*}
where $\mathscr{S}^d$ is a sparse family.

Now, let us prove the claim \eqref{irritatesparse}.
For any fixed $\mathcal{O}(B_0) \in\mathbb{R}^N$, we have
$
\mathcal{O}(B_0) \subseteq Q^{d,0} \subseteq C_0 \mathcal{O}(B_0)
$
and $1/12 B(Q^{d,0}) \subset Q^{d,0} \subset B(Q^{d,0})$.
Then, we have
$$\omega(\mathcal{O}(B_0)) \approx \omega(Q^{d,0}) \quad\text{and}\quad
 \omega(B(Q^{d,0})) \approx \omega(Q^{d,0}) \quad\Longrightarrow \quad
 \omega(\mathcal{O}(B_0)) \approx
 \omega(B(Q^{d,0})).$$
Thus, it follows that for a constant $C_{E^d}$ which will be fixed later on, we define the set $E^d$ by
\begin{align*}
E^d:= & \left\{x \in \mathcal{O}(B_0) :
\sum_{\sigma \in G}|f(\sigma(x))| > C_{E^d}|f|_{\tilde{C}_{0}B(Q^{d,0})}\right\}\\
&   \bigcup \left\{x \in \mathcal{O}(B_0) : M_{T, B(Q^{d,0})} f(x) >
    C_{E^d}|f|_{\tilde{C}_{0}B(Q^{d,0})}\right\}
\end{align*}
satisfying
\begin{align}\label{Ed}
\omega(E^d)\leq \frac{1}{4\tilde{C}_{d}}\omega(\mathcal{O}(B_0))
\end{align}
 with $\tilde{C}_{d}$ appearing in \eqref{doublingQ}. Indeed, we can achieve this. First, it is easy to check that
\begin{align*}
\omega\left( \left\{x \in \mathcal{O}(B_0) :
\sum_{\sigma \in G}|f(\sigma(x))| > C_{E^d}
|f|_{\tilde{C}_{0}B(Q^{d,0})}\right\}\right)
\lesssim \frac{\omega(\tilde{C}_{0}B(Q^{d,0}))}{C_{E^d}}
\lesssim \frac{\omega(\mathcal{O}(B_0))}{C_{E^d}}
 .
\end{align*}
By \eqref{MTB0<MT}, \eqref{eq35} and the weak $(1,1)$ bounds for $M$ and $T^*$, we have that
\begin{align*}
&\omega\left( \left\{x \in \mathcal{O}(B_0) : M_{T, B(Q^{d,0})} f(x) >
    C_{E^d}|f|_{\tilde{C}_{0}B(Q^{d,0})}\right\}\right)\\
\lesssim &
\omega\left( \left\{x \in \mathcal{O}(B_0) :\sum_{\sigma \in G}M(f\mathbf{1}_{\tilde{C}_{0}B(Q^{d,0})})(\sigma(x)) + \sum_{\sigma \in G}T^{*}(f\mathbf{1}_{\tilde{C}_{0}B(Q^{d,0})})(\sigma(x))>
    C_{E^d}|f|_{\tilde{C}_{0}B(Q^{d,0})}\right\}\right)
    \\
\lesssim & (\|M\|_{L^1 \rightarrow L^{1,\infty}}+\|T^*\|_{L^1 \rightarrow L^{1,\infty}}) \frac{\omega(\tilde{C}_{0}B(Q^{d,0}))}{C_{E^d}}
\lesssim \frac{\omega(\mathcal{O}(B_0))}{C_{E^d}}.
\end{align*}
Note that the implicit constants are finite. Therefore, we can choose $C_{E^d}$ big enough such that \eqref{Ed} holds.

Now, let us apply the Calder\'{o}n-Zygmund decomposition to the function $\mathbf{1}_{E^d}$ on $\mathcal{O}(B_0)$ at height $\frac{1}{2\tilde{C}_{d}}$.
Note that
$$\int_{\mathcal{O}(B_0)}\mathbf{1}_{E^d}(x) d\omega(x) = \omega(E^d \cap \mathcal{O}(B_0))\leq \omega(E^d)
\leq \frac{1}{2\tilde{C}_{d}}\omega(\mathcal{O}(B_0)).$$
We decompose $\mathcal{O}(B_0)$ into dyadic cubes
$\{Q^{d,1}_{j}\}_{j\in J} \subset \mathscr{D}^d(Q^{d,0})$.
If
\begin{align}\label{selectionrule}
\int_{Q^{d,1}_{j}}\mathbf{1}_{E^d}(x) d\omega(x) > \frac{1}{2\tilde{C}_{d}}
\omega(Q^{d,1}_{j}),
\end{align}
we will select it into the set $\{P^{d}_{j}\}_{j\in J}$. Otherwise, we will decompose it into dyadic cubes
$\{Q^{d,2}_{j}\}_{j\in J} \subset \mathscr{D}^d(Q^{d,1}_j)$.
Iterating this process, we get pairwise disjoint cubes
$\{P^{d}_{j}\}_{j\in J}\subset \mathscr{D}^d(Q^{d,0})$
such that
\begin{align}\label{notequal}
\frac{1}{2\tilde{C}_{d}}\omega(P^{d}_{j}) < \omega(P^{d}_{j} \cap E^d ) \leq \frac{1}{2}\omega(P^{d}_{j}),
\end{align}
where the last inequality follows from \eqref{doublingQ}.

We also get
\begin{align}\label{contained}
\omega(E^d \backslash (\cup_{j}P^{d}_{j} ))=0.
\end{align}
Moreover, from \eqref{Ed} we have
$$\sum_j \omega(P^{d}_{j})< 2\tilde{C}_{d} \omega(E^d \cap (\cup_{j}P^{d}_{j} )) =2\tilde{C}_{d} \omega(E^d) \leq \frac{1}{2}\omega(\mathcal{O}(B_0)).$$
Then, we split
\begin{equation}\label{eq8}
\begin{aligned}
& |T(f \mathbf{1}_{\tilde{C}_{0}B(Q^{d,0})})(x)| \mathbf{1}_{Q^{d,0}}(x) \\
= & |T(f \mathbf{1}_{\tilde{C}_{0}B(Q^{d,0})})(x)| \mathbf{1}_{Q^{d,0} \backslash  \cup_{j}P^d_{j}}(x)
 +  \sum_{j}|T(f \mathbf{1}_{\tilde{C}_{0}B(Q^{d,0})})(x)| \mathbf{1}_{ P^d_{j}}(x)\\
\leq &  |T(f
   \mathbf{1}_{\tilde{C}_{0}B(Q^{d,0})})(x)| \mathbf{1}_{Q^{d,0} \backslash  \cup_{j}P^d_{j}}(x)
+   \sum_{j}|T(f
   \mathbf{1}_{\tilde{C}_{0}B(Q^{d,0})
  \backslash  \tilde{C}_{0}B(P^d_{j})})(x)| \mathbf{1}_{ P^d_{j}}(x)\\
&+ \sum_{j}|T(f
  \mathbf{1}_{\tilde{C}_{0}B(P^d_{j})})(x)| \mathbf{1}_{ P^d_{j}}(x).
\end{aligned}
\end{equation}

By \eqref{notequal} and \eqref{contained}, we have $E^d \subsetneqq \cup_{j}P^d_{j} $. Observe that for any $x\in \mathcal{O}(B_0) \backslash  \cup_{j}P^d_{j}$, then $x \notin E^d$ but $x \in \mathcal{O}(B_0)$. Then from the definition of $E^d$ we have
\begin{align}\label{fx}
\sum_{\sigma \in G}|f(\sigma(x))| \leq C_{E^d}|f|_{\tilde{C}_{0}B(Q^{d,0})}
\end{align}
and
\begin{align}\label{M_Tx}
M_{T, B(Q^{d,0})} f(x) \leq
    C_{E^d}|f|_{\tilde{C}_{0}B(Q^{d,0})}.
\end{align}
Then, by \eqref{eqMB0}, \eqref{fx} and \eqref{M_Tx}, we have
\begin{equation}\label{eqfirstterm}
\begin{aligned}
 &|T(f \mathbf{1}_{\tilde{C}_{0} B(Q^{d,0}) })(x)|
 \mathbf{1}_{\mathcal{O}(B_0) \backslash  \cup_{j}P^d_{j}}(x)\\
\leq & \left\{C |G|\|T\|_{L^1 \rightarrow L^{1, \infty}} \sum_{\sigma \in G}|f(\sigma(x))|
       + M_{T, B(Q^{d,0})} f(x)\right\}\mathbf{1}_{\mathcal{O}(B_0) \backslash  \cup_{j}P^d_{j}}(x)\\
\leq & C |G| C_{E^d} (\|T\|_{L^1 \rightarrow L^{1, \infty}}+1)
|f|_{\tilde{C}_{0}B(Q^{d,0})}\mathbf{1}_{\mathcal{O}(B_0)}(x) .
\end{aligned}
\end{equation}

For each $j$, $P^d_{j} \cap (E^d)^{c} \neq \emptyset$ follows from \eqref{notequal}, then we can choose $x_0 \in P^d_{j} \cap (E^d)^{c}$, then we can see $M_{T, B(Q^{d,0})} f(x_0) $ satisfies \eqref{M_Tx}. By the definition of $M_{T, B(Q^{d,0})}$, for each $B(P^d_{j}) \subset B(Q^{d,0})$ and any $x\in P^d_{j}\subset B(P^d_{j})$ with $P^d_{j}\subset Q^{d,0}$, we have
\begin{align}\label{eqsecondterm}
|T(f
   \mathbf{1}_{\tilde{C}_{0}B(Q^{d,0})
  \backslash  \tilde{C}_{0}B(P^d_{j})})(x)| \mathbf{1}_{ P^d_{j}}(x) \leq
  C_{E^d}|f|_{\tilde{C}_{0}B(Q^{d,0})}\mathbf{1}_{Q^{d,0}}(x) .
\end{align}
Then, from \eqref{eq8}, \eqref{eqfirstterm} and \eqref{eqsecondterm}, it follows from the fact $\mathcal{O}(B_0) \subset Q^{d,0}$ that
\begin{align}\label{irritate}
 |T(f \mathbf{1}_{\tilde{C}_{0}B(Q^{d,0})})(x)| \mathbf{1}_{Q^{d,0}}(x)
\lesssim
|f|_{\tilde{C}_{0}B(Q^{d,0})}\mathbf{1}_{Q^{d,0}}(x)
+ \sum_{j}|T(f
  \mathbf{1}_{\tilde{C}_{0}B(P^d_{j})})(x)| \mathbf{1}_{ P^d_{j}}(x),
\end{align}
where the omitted constant depends on $|G|$, $N$, $\|M\|_{L^1 \rightarrow L^{1,\infty}}$, and $\|T^*\|_{L^1 \rightarrow L^{1,\infty}}$.

By iterating \eqref{irritate} for $P_j^d$, we can obtain the families $\{P_j^{d,k}\} \subset \mathscr{D}^d(Q^{d,0})$ with $\{P_j^{d,0}\}=\{Q^{d,0}\},\{P_j^{d,1}\}=\{P_j^d\}$ as defined above, and $\{P_j^{d,k}\}$ are the cubes obtained at the $k$-th stage of the iterative process for $k \geq 2$. Then,
\begin{align}\label{irritateresult}
 |T(f \mathbf{1}_{\tilde{C}_{0}B(Q^{d,0})})(x)| \mathbf{1}_{Q^{d,0}}(x)
\lesssim \sum_{k\geq 0} \sum_{j}
|f|_{\tilde{C}_{0}B(P_j^{d,k})}\mathbf{1}_{ P^{d,k}_{j}}(x).
\end{align}
Note that for each $P_j^{d,k}$, we have
$$P_j^{d,k} \backslash \cup_i P_i^{d,k+1} \subset P_j^{d,k}
\quad \text{and}\quad
\omega(P_j^{d,k} \backslash \cup_i P_i^{d,k+1}) >\frac{1}{2}
\omega(P_j^{d,k} ).
$$
Moreover, $\{P_j^{d,k} \backslash \cup_i P_i^{d,k+1}\}_{j,k}$ is a collection of disjoint sets.
Now, let $\mathscr{P}^d$ be the union of the families
$\{P_j^{d,k}\}_{j}$ for $k=0,1, \ldots$. We have $\mathscr{P}^d$ is a $\frac{1}{2}$-sparse family. Then, we can write \eqref{irritateresult} as
\begin{align*}
 |T(f \mathbf{1}_{\tilde{C}_{0}B(Q^{d,0})})(x)| \mathbf{1}_{Q^{d,0}}(x)
\lesssim \sum_{Q^d \in \mathscr{P}^d}
|f|_{\tilde{C}_{0}B(Q^{d})}\mathbf{1}_{ Q^{d}}(x).
\end{align*}
Thus, we prove the claim.
\end{proof}

\begin{proof}[Proof of \eqref{sparsebT} in Theorem \ref{sparseforT}]
We continue to use the notations in the proof of \eqref{sparseT} in Theorem \ref{sparseforT}. Recall that
for each $\mathcal{O}(B_0) \subseteq \mathbb{R}^N$, we have $
\mathcal{O}(B_0) \subseteq Q^{d,0} \subseteq C_0 \mathcal{O}(B_0)
$
and $1/12 B(Q^{d,0}) \subset Q^{d,0} \subset B(Q^{d,0})=\mathcal{O}(B(x_{Q^{d,0}},2\ell(Q^{d,0})))$. There also exists $Q^d_{B(Q^{d,0})}$ such that $\tilde{C}_{0}B(Q^{d,0}) \subset Q^d_{B(Q^{d,0})}\subset C_0\tilde{C}_{0}B(Q^{d,0}) $.
First, similar to \eqref{eq8}, we split $[b,T]$:
\begin{align*}
& |[b,T](f \mathbf{1}_{\tilde{C}_{0}B(Q^{d,0})})(x)| \mathbf{1}_{Q^{d,0}}(x) \\
\leq &  |[b,T](f
   \mathbf{1}_{\tilde{C}_{0}B(Q^{d,0})})(x)| \mathbf{1}_{Q^{d,0} \backslash  \cup_{j}P^d_{j}}(x)
+   \sum_{j}|[b,T](f
   \mathbf{1}_{\tilde{C}_{0}B(Q^{d,0})
  \backslash  \tilde{C}_{0}B(P^d_{j})}(x)| \mathbf{1}_{ P^d_{j}}(x)\\
&+ \sum_{j}|[b,T](f
  \mathbf{1}_{\tilde{C}_{0}B(P^d_{j})})(x)| \mathbf{1}_{ P^d_{j}}(x)\\
\leq & |b(x)-b_{Q^d_{B(Q^{d,0})}}| |T(f
   \mathbf{1}_{\tilde{C}_{0}B(Q^{d,0})})(x)| \mathbf{1}_{Q^{d,0} \backslash  \cup_{j}P^d_{j}}(x)\\
 &  +|T((b-b_{Q^d_{B(Q^{d,0})}})f
   \mathbf{1}_{\tilde{C}_{0}B(Q^{d,0})})(x)| \mathbf{1}_{Q^{d,0} \backslash  \cup_{j}P^d_{j}}(x)\\
&+    |b(x)-b_{Q^d_{B(Q^{d,0})}}|\sum_{j}|T(f
   \mathbf{1}_{\tilde{C}_{0}B(Q^{d,0})
  \backslash  \tilde{C}_{0}B(P^d_{j})})(x)| \mathbf{1}_{ P^d_{j}}(x)\\
  &+  \sum_{j} |T((b-b_{Q^d_{B(Q^{d,0})}})f
   \mathbf{1}_{\tilde{C}_{0}B(Q^{d,0})
  \backslash  \tilde{C}_{0}B(P^d_{j})})(x)| \mathbf{1}_{ P^d_{j}}(x)\\
&+ \sum_{j}|[b,T](f
  \mathbf{1}_{\tilde{C}_{0}B(P^d_{j})})(x)| \mathbf{1}_{ P^d_{j}}(x)\\
=:& \mathscr{C}_1(x)
   +\mathscr{C}_2(x)
   +\mathscr{C}_3(x)
   +\mathscr{C}_4(x)
   +\mathscr{C}_5(x).
\end{align*}
For $\mathscr{C}_1(x)
+\mathscr{C}_3(x)$, by using \eqref{eqfirstterm} and \eqref{eqsecondterm}, we have
\begin{align*}
&\mathscr{C}_1(x)
+\mathscr{C}_3(x)\\
=& |b(x)-b_{Q^d_{B(Q^{d,0})}}| \left\{|T(f
   \mathbf{1}_{\tilde{C}_{0}B(Q^{d,0})})(x)| \mathbf{1}_{Q^{d,0} \backslash  \cup_{j}P^d_{j}}(x)
   +    \sum_{j}|T(f
   \mathbf{1}_{\tilde{C}_{0}B(Q^{d,0})
  \backslash  \tilde{C}_{0}B(P^d_{j})})(x)| \mathbf{1}_{ P^d_{j}}(x)\right\}\\
\lesssim &
   |b(x)-b_{Q^d_{B(Q^{d,0})}}|
   |f|_{\tilde{C}_{0}B(Q^{d,0})}\mathbf{1}_{Q^{d,0}}(x) .
\end{align*}
To estimate  $\mathscr{C}_2(x)
+\mathscr{C}_4(x)$, let $f$ in \eqref{eqfirstterm} and \eqref{eqsecondterm} be $(b-b_{Q^d_{B(Q^{d,0})}})f$ which is also supported on $B_0$. Then, we have
\begin{align*}
\mathscr{C}_2(x)
+\mathscr{C}_4(x)
\lesssim
|(b-b_{Q^d_{B(Q^{d,0})}})f|_{\tilde{C}_{0}B(Q^{d,0})}\mathbf{1}_{Q^{d,0}}(x) .
\end{align*}
Now, we can conclude that
\begin{align*}
& |[b,T](f \mathbf{1}_{\tilde{C}_{0}B(Q^{d,0})})(x)| \mathbf{1}_{Q^{d,0}}(x) \\
\lesssim &
\left\{
|f|_{\tilde{C}_{0}B(Q^{d,0})}|b(x)-b_{Q^d_{B(Q^{d,0})}}|+
|(b-b_{Q^d_{B(Q^{d,0})}})f|_{\tilde{C}_{0}B(Q^{d,0})} \right\}\mathbf{1}_{Q^{d,0}}(x)\\
&+ \sum_{j}|[b,T](f
  \mathbf{1}_{\tilde{C}_{0}B(P^d_{j})})(x)| \mathbf{1}_{ P^d_{j}}(x).
\end{align*}
Similar to the iterative process in the proof of \eqref{sparseT} in Theorem \ref{sparseforT}, there exists a $1/2$-sparse family $\mathscr{P}^d$ such that
\begin{align*}
& |[b,T](f \mathbf{1}_{\tilde{C}_{0}B(Q^{d,0})})(x)| \mathbf{1}_{Q^{d,0}}(x) \\
\lesssim &
\sum_{Q^d \in \mathscr{P}^d}
\left\{
|f|_{\tilde{C}_{0}B(Q^{d})}|b(x)-b_{Q^d_{B(Q^{d,0})}}|+
|(b-b_{Q^d_{B(Q^{d,0})}})f|_{\tilde{C}_{0}B(Q^{d})} \right\}\mathbf{1}_{Q^{d}}(x).
\end{align*}
This the the corresponding version of the claim \eqref{irritatesparse} for $[b,T]$. Then we apply the above result to each $\mathcal{O}(B_{i,t}) \in \{\mathcal{O}(B_{i,t})\}_{i,t}$ and get the $1/2$-sparse family $\mathscr{P}^d_{i,t}$. Let $\mathscr{F}^d$ the union of $\{\mathscr{P}^d_{i,t}\}_{i,t}$, for each $Q^d \in \mathscr{F}^d$, we have $ Q^d_{B(Q^{d})}\subset C_0\tilde{C}_{0}B(Q^{d}) $ with $|f|_{\tilde{C}_{0}B(Q^{d}) } \lesssim |f|_{Q^d_{B(Q^{d})}}$. Then, we have
\begin{align*}
 |[b,T](f )(x)|
\lesssim &
\sum_{Q^d \in \mathscr{F}^d}
\left\{
|f|_{\tilde{C}_{0}B(Q^{d})}|b(x)-b_{Q^{d}_{B(Q^{d})}}|+
|(b-b_{Q^{d}_{B(Q^{d})}})f|_{\tilde{C}_{0}B(Q^{d})} \right\}\mathbf{1}_{Q^{d}}(x)\\
\lesssim &
\sum_{Q^d \in \mathscr{F}^d}\left\{
|f|_{Q^{d}_{B(Q^{d})}}|b(x)-b_{Q^{d}_{B(Q^{d})}}|+
|(b-b_{Q^{d}_{B(Q^{d})}})f|_{Q^{d}_{B(Q^{d})}} \right\}\mathbf{1}_{Q^{d}}(x).
\end{align*}

Then, we have that there exists a sparse family $\mathscr{S}^d=\{Q^{d}_{B(Q^{d})}: Q^d \in \mathscr{F}^d\}$ such that for any $x\in \mathbb{R}^N$,
\begin{align*}
 |[b,T]f (x)|
\lesssim \sum_{Q^d \in \mathscr{S}^d}
\left\{
|f|_{Q^{d}}|b(x)-b_{Q^{d}}| +
|(b-b_{Q^{d}})f|_{Q^{d}} \right\}\mathbf{1}_{Q^{d}}(x).
\end{align*}
\end{proof}

\section{Proof of Theorem \ref{twoweightbound}}\label{Sec 4}

In this section, we shall prove Theorem \ref{twoweightbound}. First, to get the weighted bounds for $T$ with Dunkl-Muckenhoupt weight, we will use the sparse domination in \eqref{weightT} in Theorem \ref{sparseforT}.

Given any $\theta$-sparse family with $0<\theta<1$, we define
$$
\mathcal{A}(f)(x):=\sum_{Q^d \in \mathscr{S}^d }|f|_{Q^d}
         \mathbf{1}_{Q^d}(x)
$$
on $(\mathbb{R}^N, d, d \omega)$ which is a self-adjoint operator.

For any $u \in A_{p}^d$, $1<p<\infty$, applying the proof of Theorem 3.1 on Page 461 and 462 in \cite{M1}, by \eqref{wp} and Lemma \ref{Muweight}, we have
\begin{align}\label{sparsedifi}
\left\|\mathcal{A}(f)\right\|_{L^p(\mathbb{R}^N,d,u d\omega)}
\leq C
   [u]_{A_{p}^d}^{\max\{p'/p,1\}} \|f\|_{L^p(\mathbb{R}^N,d,u d\omega)},
\end{align}
where $C$ depends on $p$.
Then, by using \eqref{sparseT} in Theorem \ref{sparseforT}, we get $T$ is bounded on $L^p(\mathbb{R}^N,d,u d\omega)$ for any $u\in A^d_p$ and $1<p<\infty$.

Next, we will give the two-weight bounds for the commutator $[b,T]$ by following the ideas in \cite{AMR}, \cite{DGKLWY} and \cite{Lerner2016}.

In the Dunkl setting, adapting the method of \cite[Lemma 3.5]{DGKLWY}, we can obtain the following result.
\begin{lemma}(\cite{DGKLWY})\label{lemmab}
Let $\mathscr{S}^d$ be a sparse family contained in a dyadic system $\mathscr{D}^d,$ $\vartheta$ is a weight, $b \in \mathrm{BMO}_{\vartheta}^d$. There exists a possibly larger sparse family $\tilde{\mathscr{S}}^d \subset \mathscr{D}^d$ containing $\mathscr{S}^d$ such that, for every $Q^d \in \tilde{\mathscr{S}}^d$
$$
|b(x)-b_{Q^d}| \lesssim \sum_{P^d \in \tilde{\mathscr{S}}^d, P^d \subset Q^d } \frac{1}{\omega(P^d)} \int_{P^d}|b(y)-b_{P^d}| d\omega (y)\mathbf{1}_{P^d}(x).
$$
\end{lemma}
Then, it is easy to see that
\begin{equation}\label{bA}
\begin{aligned}
&\int_{Q^d}|b(x)-b_{Q^d}||f(x)|d \omega (x)\\
\lesssim & \int_{Q^d}\sum_{P^d \in \tilde{\mathscr{S}}^d, P^d \subset Q^d } \frac{1}{\omega(P^d)} \int_{P^d}|b(y)-b_{P^d}| d\omega (y)\mathbf{1}_{P^d}(x) |f(x)|d \omega (x)\\
\leq & \|b\|_{\mathrm{BMO}_{\vartheta}^d}
\sum_{P^d \in \tilde{\mathscr{S}}^d, P^d \subset Q^d }
          \frac{\vartheta(P^d)}{\omega(P^d)}
\int_{Q^d}
          |f(x)|\mathbf{1}_{P^d}(x)d\omega (x)\\
= & \|b\|_{\mathrm{BMO}_{\vartheta}^d}\int_{Q^d}
          \mathcal{A}(|f|)(x)\vartheta(x)d\omega (x).
\end{aligned}
\end{equation}
By \eqref{sparsebT} in Theorem \ref{sparseforT}, we have
\begin{align*}
 &\|[b,T]f\|_{L^{p}(\mathbb{R}^N,d,v d\omega)}\\
\lesssim &\bigg\|\sum_{Q^d \in \mathscr{S}^d}
|f|_{Q^{d}}|b-b_{Q^{d}}|\mathbf{1}_{Q^{d}} \bigg\|_{L^{p}(\mathbb{R}^N,d,v d\omega)} +
\bigg\|\sum_{Q^d \in \mathscr{S}^d}|(b-b_{Q^{d}})f|_{Q^{d}} \mathbf{1}_{Q^{d}}\bigg\|_{L^{p}(\mathbb{R}^N,d,v d\omega)}\\
=: & \mathcal{B}_{1} +\mathcal{B}_{2}.
\end{align*}

By using \eqref{sparsedifi} and \eqref{bA}, note that $\vartheta =(u/v)^{1/p}$, we have the estimate
\begin{align*}
\mathcal{B}_{1} =& \sup_{\|g\|_{L^{p'}(\mathbb{R}^N,d,v d\omega)}=1}\left|\int_{\mathbb{R}^N}\sum_{Q^d \in \mathscr{S}^d}
|f|_{Q^{d}}|b(x)-b_{Q^{d}}|\mathbf{1}_{Q^{d}}(x) g(x)v(x) d\omega(x)\right|\\
\leq &\sup_{\|g\|_{L^{p'}(\mathbb{R}^N,d,v d\omega)}=1}
\sum_{Q^d \in \mathscr{S}^d}
|f|_{Q^{d}}\int_{Q^{d}}|b(x)-b_{Q^{d}}| |g(x)|v(x) d\omega(x)\\
\lesssim & \sup_{\|g\|_{L^{p'}(\mathbb{R}^N,d,v d\omega)}=1}\sum_{Q^d \in \mathscr{S}^d}|f|_{Q^{d}}
 \|b\|_{\mathrm{BMO}_{\vartheta}^d}
\int_{Q^d}  \mathcal{A}(|g|v)(x)  \vartheta(x)d\omega(x) \\
= & \sup_{\|g\|_{L^{p'}(\mathbb{R}^N,d,v d\omega)}=1}
 \|b\|_{\mathrm{BMO}_{\vartheta}^d}
\int_{\mathbb{R}^N}   \left( \mathcal{A}(|f|)(x) \vartheta(x)\right) \mathcal{A}(|g|v)(x)d\omega(x)\\
\leq &\|b\|_{\mathrm{BMO}_{\vartheta}^d}\sup_{\|g\|_{L^{p'}(\mathbb{R}^N,d,v d\omega)}=1}
\|g\|_{L^{p'}(\mathbb{R}^N,d,v d\omega)}
\| \mathcal{A}( \mathcal{A}(|f|) \vartheta)\|_{L^{p}(\mathbb{R}^N,d,v d\omega)}\\
\lesssim & \|b\|_{\mathrm{BMO}_{\vartheta}^d}[v]_{A_{p}^d}^{\max\{p'/p,1\}}[u]_{A_{p}^d}^{\max\{p'/p,1\}}
\| f\|_{L^{p}(\mathbb{R}^N,d,u d\omega)}.
\end{align*}

Similarly, we have the estimate for $\mathcal{B}_{2}$,
\begin{align*}
\mathcal{B}_{2} =& \sup_{\|g\|_{L^{p'}(\mathbb{R}^N,d,v d\omega)}=1}\left|\int_{\mathbb{R}^N}\sum_{Q^d \in \mathscr{S}^d} \frac{1}{\omega(Q^d)}\int_{Q^d}
|b(y)-b_{Q^{d}}||f(y)|d\omega(y)\mathbf{1}_{Q^{d}}(x) g(x)v(x) d\omega(x)\right|\\
\lesssim & \sup_{\|g\|_{L^{p'}(\mathbb{R}^N,d,v d\omega)}=1}
 \int_{\mathbb{R}^N}\sum_{Q^d \in \mathscr{S}^d} \frac{1}{\omega(Q^d)}\|b\|_{\mathrm{BMO}_{\vartheta}^d}\int_{Q^d} \mathcal{A}(|f|)(y)
 \vartheta(y)d\omega(y)\mathbf{1}_{Q^{d}}(x) |g(x)|v(x) d\omega(x)\\
= & \sup_{\|g\|_{L^{p'}(\mathbb{R}^N,d,v d\omega)}=1}
 \|b\|_{\mathrm{BMO}_{\vartheta}^d} \int_{\mathbb{R}^N}
 \mathcal{A}(\mathcal{A}(|f|) \vartheta)(x) |g(x)|v(x) d\omega(x)\\
\lesssim & \|b\|_{\mathrm{BMO}_{\vartheta}^d}[v]_{A_{p}^d}^{\max\{p'/p,1\}}[u]_{A_{p}^d}^{\max\{p'/p,1\}}
\| f\|_{L^{p}(\mathbb{R}^N,d,u d\omega)}.
\end{align*}
Thus, the proof of \eqref{weightbT} in Theorem \ref{twoweightbound} is complete.

Now, let us prove the result in Remark \ref{BMOdunklremark}. First, let us give a definition.
\begin{definition}
Let $b$ be a real-valued measurable function.
For $B \subseteq \mathbb{R}^N$ with $\omega(B)<\infty$, we define a median value $m_b(B)$ of $b$ over $B$ to be a real number satisfying
$$
\omega\left(\left\{x \in B: b(x)>m_b(B)\right\}\right) \leq \frac{1}{2} \omega(B) \quad {\rm{ and }} \quad \omega\left(\left\{x \in B: b(x)<m_b(B)\right\}\right) \leq \frac{1}{2} \omega(B) .
$$
\end{definition}
For given $b \in L_{loc}^1(\mathbb{R}^N, d \omega)$ and for any ball $B$, denote by $\Omega(b, B)$ the oscillation
$$
\Omega(b, B):=\frac{1}{\omega(B)} \int_B|b(x)-b_B| d \omega(x).
$$
Let $R_j$, $j=1,\ldots,N$ be the Dunkl-Riesz transform. Let $B_0:=B(x_0,r)$ be any ball centred at $x_0$ with radius $r>0$.
Then, we choose $\tilde{B}_0=B\left(\tilde{x}_0, r\right)$ with $\|\tilde{x}_0-x_0\|=5r$ such that $y_j-x_j \geq r$ and $\|x-y\| \approx r$ for $x \in B_0$ and $y \in \tilde{B}_0$.

Now, we choose two measurable sets
$$
E_1 \subseteq\left\{y \in \tilde{B}_0: b(y)<m_b(\tilde{B}_0)\right\} \quad \text { and } \quad E_2 \subseteq \left\{y \in \tilde{B}_0: b(y) \geq m_b(\tilde{B}_0)\right\}
$$
such that $\omega(E_i)=\frac{1}{2} \omega(\tilde{B}_0)$, $i=1,2$, and that $E_1 \cup E_2=\tilde{B}_0$, $E_1 \cap E_2=\emptyset$.
Moreover, we define
$$
B_1:=\left\{x \in B_0: b(x) \geq m_b(\tilde{B}_0)\right\} \quad \text { and } \quad B_2:=\left\{x \in B_0: b(x) \leq m_b(\tilde{B}_0)\right\}
$$

Now based on the definitions of $E_i$ and $B_i$, we have
\begin{align*}
& b(x)\geq m_b (\tilde{B}_0 ) > b(y), \quad(x, y) \in B_{1} \times E_{1} ;\\
& b(x)\leq m_b (\tilde{B}_0 ) \leq b(y), \quad(x, y) \in B_{2} \times E_{2}.
\end{align*}
Thus, for all $(x, y) \in B_i \times E_i, i=1,2$, we have that $b(x)-b(y)$ does not change sign and that
\begin{align*}
|b(x)-b(y)| & =\left|b(x)-m_b(\tilde{B}_0)+m_b(\tilde{B}_0)-b(y)\right| \\
& =\left|b(x)-m_b(\tilde{B}_0)\right|+\left|m_b(\tilde{B}_0)-b(y)\right| \geq  \left|b(x)-m_b(\tilde{B}_0)\right|.
\end{align*}
It is easy to check that
$$
\Omega(b, B_0)\leq
\frac{2}{\omega(B_0)} \int_{B_0} \left|b(x)-m_b(\tilde{B}_0)\right| d \omega(x).
$$
From \cite[Theorem 1.2]{HLLW}, we know that for every $(x,y)\in B_0 \times \tilde{B}_0$, the convolution kernel of the Dunkl-Riesz transform $R_j(x,y)$ does not change sign and has lower bounds
$$
\left|R_j(x, y)\right| \gtrsim \frac{1}{\omega(B(x_0, r))}.
$$
Let $f_i=\mathbf 1_{E_i}$ for $i=1,2$. Since $\|\tilde{x}_0-x_0\|=5r$, we have $\omega(B_0)\approx \omega(\tilde{B}_0)$. By using above estimates, we have
\begin{align*}
 \frac{1}{\omega(B_0)} \sum_{i=1}^2 \int_{B_0}  \left|[b, R_j] f_i(x)\right| d \omega(x)
\gtrsim &\left |\Omega(b, B_0)\right| .
\end{align*}

Recall that $u=\vartheta^p v$. Since $u$ is a radial weight, then $u$ satisfies the reverse H\"{o}lder inequality \eqref{reverseHolder}. Applying \cite[Theorem 7.3.3]{GrafakosCla2014} and the argument of \cite[Page 109 and 115]{Lerner2019}, we obtain
\begin{equation}\label{reverseu}
\begin{aligned}
\frac{1}{\omega(B_0)}\int_{B_0}    u(x) d \omega(x)\leq& \left(\frac{1}{\omega(B_0)}\int_{B_0}    u^{1/(p+1)}(x) d \omega(x) \right)^{p+1}\\
 \leq &
\left(\frac{1}{\omega(B_0)}\int_{B_0}   \vartheta(x) d \omega(x)\right)^{ p}
\left(\frac{1}{\omega(B_0)}\int_{B_0}   v(x) d \omega(x)\right).
\end{aligned}
\end{equation}
Next, from H\"{o}lder's inequality and the weighted boundedness of $[b, R_j]$ which is a spacial case of \eqref{weightbT} in Theorem \ref{twoweightbound}, by \eqref{reverseu}, we deduce that
\begin{align*}
&\frac{1}{\omega(B_0)} \sum_{i=1}^2 \int_{B_0} \left|[b, R_j] f_i(x)\right| d \omega(x)\\
\lesssim & \frac{1}{\omega(B_0)} \sum_{i=1}^2
\left(\int_{B_0}   \left|[b, R_j] f_i(x)\right|^p v(x) d \omega(x)\right)^{1 / p}
\left(\int_{B_0}   v^{-1/(p-1)}(x) d \omega(x)\right)^{(p-1) / p} \\
\lesssim &  \left\|[b, R_j]\right\|_{L^p(\mathbb{R}^N, d,ud\omega) \rightarrow L^p (\mathbb{R}^N,d, vd\omega)}
\frac{1}{\omega(B_0)}
\sum_{i=1}^2
\left(\int_{E_i}    u(x) d \omega(x)\right)^{1 / p} \left(\int_{B_0}   v^{-1/(p-1)}(x) d \omega(x)\right)^{(p-1) / p}\\
\lesssim &  \left\|[b, R_j]\right\|_{L^p(\mathbb{R}^N, d,ud\omega) \rightarrow L^p (\mathbb{R}^N,d, vd\omega)}
\left(\frac{1}{\omega(B_0)}\int_{B_0}    u(x) d \omega(x)\right)^{1 / p} \left(\frac{1}{\omega(B_0)}\int_{B_0}   v^{-1/(p-1)}(x) d \omega(x)\right)^{(p-1) / p}\\
\lesssim & \left\|[b, R_j]\right\|_{L^p(\mathbb{R}^N, d,ud\omega) \rightarrow L^p (\mathbb{R}^N,d, vd\omega)} [v]_{A_p^d}^{1/p}
\frac{\vartheta(B_0)}{\omega(B_0)}.
\end{align*}
For any $B_0 \in \mathbb{R}^N$, combining with $\omega(B_0)\approx \omega(\tilde{B}_0)\geq \omega(E_i)$, we have
\begin{align*}
&\frac{1}{\vartheta(B_0)} \int_{B_0}|b(x)-b_{B_0}| d \omega(x)\\
=&\frac{\omega(B_0)}{\vartheta(B_0)}|\Omega(b, B_0)|\\
\lesssim & \frac{\omega(B_0)}{\vartheta(B_0)}\frac{1}{\omega(B_0)} \sum_{i=1}^2 \int_{B_0} \left|[b, R_j] f_i(x)\right| d \omega(x)\\
\lesssim & \left\|[b, R_j]\right\|_{L^p(\mathbb{R}^N, d,ud\omega) \rightarrow L^p (\mathbb{R}^N,d, vd\omega)}
[v]_{A_p^d}^{1/p}
\frac{\vartheta(B_0)}{\omega(B_0)} \frac{\omega(B_0)}{\vartheta(B_0)}\\
=& \left\|[b, R_j]\right\|_{L^p(\mathbb{R}^N, d,ud\omega) \rightarrow L^p (\mathbb{R}^N,d, vd\omega)}
[v]_{A_p^d}^{1/p}.
\end{align*}
The proof is complete.

\section{Proof of Theorem \ref{Extrapolation}}\label{proofExtrapolation}
First, we give the definition of two extrapolation methods. See \cite{KokiMastMesk2020} for more details. Given a family of Banach function spaces $\left\{X_\alpha\right\}_{\alpha \in \mathcal{I}}$ which satisfies: there exists a Banach space $X \subset X_\alpha$ for all $\alpha \in \mathcal{I}$ and $\sup _{\alpha \in \mathcal{I}} \|$ id: $X \rightarrow X_\alpha \|<\infty$. Then, we define
$$
\Delta\left(\left\{X_\alpha\right\}_{\alpha \in \mathcal{I}}\right):=\left\{f \in \bigcap_{\alpha \in \mathcal{I}} X_\alpha :\|f\|_{\Delta\left(\{X_\alpha\}_{\alpha \in \mathcal{I}}\right)}:=\sup _{\alpha \in \mathcal{I}}\|f\|_{X_\alpha}<\infty\right\} .
$$
Note that $\Delta\left(\left\{X_\alpha\right\}_{\alpha \in \mathcal{I}}\right)$ is a Banach function space equipped with the norm $\|\cdot\|_{\Delta\left(\{X_\alpha\}_{\alpha \in \mathcal{I}}\right)}$.

Given a family of Banach function spaces $\left\{X_\alpha\right\}_{\alpha \in \mathcal{I}}$ satisfies that there exists a Banach space $Y$ such that $X_\alpha \hookrightarrow Y$ and $\sup _{\alpha \in \mathcal{I}}\left\|\mathrm{id}: X_\alpha \rightarrow Y\right\|<\infty$, then we define $\Sigma\left(\left\{X_\alpha\right\}_{\alpha \in \mathcal{I}}\right)$ to be a space of all $f \in \mathcal{X}$ with the representation of the form
\begin{align}\label{representation}
f=\sum_{\alpha \in \mathcal{I}} f_\alpha \quad \text { where } f_\alpha \in X_\alpha\text { and }\sum_\alpha\left\|f_\alpha\right\|_{X_\alpha}<\infty.
\end{align}
From our hypothesis \eqref{representation}, we know that there are countably many $f_\alpha\neq 0$ in $\sum_{\alpha \in \mathcal{I}} f_\alpha$. Moreover, we know that $\sum_{\alpha \in \mathcal{I}} f_\alpha$ converges in $Y$. Thus, the space $\Sigma\left(X_\alpha\right)$ is the Banach space with the norm
$$
\|f\|_{\Sigma\left(\left\{X_\alpha\right\}_{\alpha \in \mathcal{I}}\right)}:=\inf \left\{\sum_{\alpha \in \mathcal{I}}\left\|f_\alpha\right\|_{X_\alpha} : f=\sum_{\alpha \in \mathcal{I}} f_\alpha\right\},
$$
where the infimum is taken over all the possible representations of $f$ in the form \eqref{representation}.

We borrow some ideas from \cite{KokiMastMesk2020} to prove Theorem \ref{Extrapolation}.
By Theorem \ref{sparseforT} and \eqref{wp}, we have that there exists a spares family $\mathscr{S}^d \subset \mathscr{D}^d$ such that
 \begin{align}\label{eqTM}
\int_{\mathbb{R}^N}|T f(x)| u(x) d \omega(x) \leq C [u]_{A_{p}} \int_{\mathbb{R}^N}M_d f(x)u(x) d \omega(x)
\end{align}
holds for any function $f$ with compact support and $u \in A^d_{p}$ for $1<p<\infty$.

First, for any Banach function space $X$, we assume that $M_d$ is bounded in $X$ and $X'$. We give the construction of the mapping $\varphi$ on $X'$ as in \cite{Rubio1982} and \cite[Lemma 7.5.4]{GrafakosCla2014}:
\begin{align}\label{eqDefR}
\varphi:=\sum_{k=0}^{\infty}\frac{M_d^{k}g}{(2\|M_d\|_{X'})^k}
\end{align}
for any $g\in X'$ and we define $M_d^0g:=|g|$.
By the assumption that $M_d$ is bounded in $X$ and $X'$ and by the definition of $A_1^d$, $\varphi$ satisfies the properties:
\begin{equation}\label{Rg<g}
\|\varphi (g)\|_{X'}\leq 2\|g\|_{X'},
\end{equation}
\begin{equation}\label{g<Rg}
|g|\leq \varphi(g),
\end{equation}
\begin{equation}\label{RgA1}
\varphi(g) \in A_1^d \quad \text{with} \quad [\varphi(g)]_{A_1^d}\leq 2\|M_d\|_{X'}.
\end{equation}
Then, by \eqref{wpq}, we have $\varphi(g) \in A_p^d$ and $[\varphi(g)]_{A_p^d}\leq [\varphi(g)]_{A_1^d}$ for any $1<p<\infty$.

By the assumptions of the theorem, using \eqref{Rg<g}-\eqref{RgA1}, \eqref{wpq} and \eqref{eqTM}, we have for any $f\in X$, there exists $g \in X'$ such that,
\begin{equation}\label{TXalpha}
\begin{aligned}
\|Tf\|_{X}= & \sup_{\|g\|_{X'}=1}\left|\int_{\mathbb{R}^N}Tf(x)g(x)d\omega(x)\right|\\
\leq& \sup_{\|g\|_{X'}=1} \int_{\mathbb{R}^N}|Tf(x)|\varphi(g)(x)d\omega(x)\\
\leq& C\sup_{\|g\|_{X'}=1} [\varphi(g)]_{A_1}\int_{\mathbb{R}^N}M_df(x)\varphi(g)(x)d\omega(x)\\
\leq& C\sup_{\|g\|_{X'}=1} \|M_d\|_{X'}\|M_df\|_{X}\|\varphi(g)\|_{X'}\\
\leq& C\sup_{\|g\|_{X'}=1} \|M_d\|_{X'}\|M_d\|_{X}\|f\|_{X}\|g\|_{X'}\\
\leq& C \|M_d\|_{X'}\|M_d\|_{X}\|f\|_{X}.
\end{aligned}
\end{equation}
Observing \eqref{TXalpha}, to get $T$ is bounded on
$$\Delta(\{X_{\alpha}\}_{\alpha \in \mathcal{I}})\quad\text{ and }\quad
\Sigma(\{X_{\alpha}\}_{\alpha \in \mathcal{I}}),$$
 it is sufficient to prove that $M_d$ is bounded on
 \begin{align}\label{spaces}
 \Delta(\{X_{\alpha}\}_{\alpha \in \mathcal{I}}),\quad
(\Delta(\{X_{\alpha}\}_{\alpha \in \mathcal{I}}))',\quad
\Sigma(\{X_{\alpha}\}_{\alpha \in \mathcal{I}})\quad\text{ and }\quad
(\Sigma(\{X_{\alpha}\}_{\alpha \in \mathcal{I}}))'.
\end{align}
 On the one hand, it is easy to see that $M_d$ is bounded on
$\Delta(\{X_{\alpha}\}_{\alpha \in \mathcal{I}})$, $\Sigma(\{X_{\alpha}\}_{\alpha \in \mathcal{I}})$, $\Delta(\{X'_{\alpha}\}_{\alpha \in \mathcal{I}})$ and $\Sigma(\{X'_{\alpha}\}_{\alpha \in \mathcal{I}})$ from the assumptions of Theorem \ref{Extrapolation}. The first on and the third one in \eqref{spaces} are proved.

 On the other hand, it is known that (see \cite{KokiMastMesk2020})
\begin{align}\label{eqDelta}
(\Delta(\{X_{\alpha}\}_{\alpha \in \mathcal{I}}))'
\cong
((\Sigma(\{X'_{\alpha}\}_{\alpha \in \mathcal{I}}))')'
\end{align}
and that by the duality in \cite[Lemma 1]{MasSan2014}
$$
(\Sigma(\{X_{\alpha}\}_{\alpha \in \mathcal{I}}))'
\cong
\Delta(\{X'_{\alpha}\}_{\alpha \in \mathcal{I}}).
$$
Here, we write $X \cong Y$ if $X=Y$ and the identity map id: $X \rightarrow Y$ is an isometry onto.
Then, the second one and the forth one in \eqref{spaces} are proved.
The proof is complete.

\section*{Acknowledgements}
Yanping Chen was supported by the National Natural Science Foundation of China (Grant numbers [12371092], [12326366] and [12326371]).

\end{document}